\documentclass[12pt, reqno]{amsart}
\usepackage{amssymb,amsmath,amsopn,textcomp, thmtools}

\usepackage[backgroundcolor=white, bordercolor=blue, 
linecolor=blue]{todonotes}
\usepackage[normalem]{ulem}

\usepackage[a4paper,margin=1in,footskip=0.25in]{geometry}
\usepackage[ansinew]{inputenc}
\usepackage[english]{babel}
\usepackage{mathrsfs}
\usepackage{setspace}
\usepackage{stmaryrd}
\usepackage[titletoc]{appendix}
\usepackage[colorlinks=true, linkcolor=blue, citecolor=blue]{hyperref}

\usepackage{dsfont}
\usepackage{bbm}

\usepackage{graphicx}
\usepackage{color}
\usepackage{caption,rotating}

\newtheorem{theorem}{Theorem}[section]

\newtheorem{proposition}[theorem]{Proposition}
\newtheorem{lemma}[theorem]{Lemma}
\newtheorem*{theorem*}{Theorem}

\theoremstyle{definition}
\newtheorem{definition}[theorem]{Definition}
\newtheorem{remark}[theorem]{Remark}
\newtheorem*{remark*}{Remark}

\renewenvironment{proof}[1][Proof]{\noindent\textit{#1.} }{\hfill 
	\rule{0.5em}{0.5em}}

\addto\extrasenglish{}
\addto\extrasenglish{} 

\numberwithin{equation}{section}

\usepackage{color}

\newcommand{\N}{\mathbb{N}}
\newcommand{\R}{\mathbb{R}}

\newcommand{\Rd}{{\mathbb{R}^d}}

\newcommand{\Z}{\mathbb Z}
\newcommand{\E}{\mathbb E}
\newcommand{\Pp}{\mathbb P}
\newcommand{\1}{\mathbbm 1}

\newcommand{\I}{{\rm I}}
\newcommand{\II}{{\rm II}}

\def\nn{\nonumber}

\newcommand{\switch}{\epsilon}

\newcommand{\HHq}[2]{\big(H_{#1}^{#2}\big)}

\newcommand{\Tj}[2]{\sum_{j={#1}}^{#2} \lengthR_j}
\newcommand{\TT}[2]{\Theta({#1},{#2})}
\newcommand{\gj}[2]{\sum_{j={#1}}^{#2} \gamma_j}
\newcommand{\gam}[2]{\Upsilon({#1},{#2})}
\newcommand{\lengthR}{\theta}

\renewcommand{\d}{\mathrm{d}}

\newcommand{\cSi}{\mathcal{S}^i} 

\newcommand{\cTi}{\mathcal{T}^i} 

\author{Moritz Kassmann}
\email{moritz.kassmann@uni-bielefeld.de}
\address{Fakult\"{a}t f\"{u}r Mathematik, Postfach 100131, D-33501 Bielefeld, Germany}

\author{Kyung-Youn Kim}
\email{kyungyoun07@gmail.com}
\address{Department of Applied Mathematics, National Chung Hsing University, Taichung 402, Taiwan}

\author{Takashi Kumagai}
\email{kumagai@kurims.kyoto-u.ac.jp}
\address{Research Institute for Mathematical Sciences
	Kyoto University, Kyoto 606-8502, Japan}

\title[Heat kernel bounds for systems of nonlocal equations]
{Heat kernel bounds for \\ nonlocal operators with singular kernels}
\subjclass[2010]{Primary 60J75; Secondary 35K08}
\keywords{Markov jump process, heat kernel, integro-differential operator}

\begin{document}
	
\thanks{Financial support of the German Science Foundation through the International Research Training Group Bielefeld-Seoul IRTG 2235,  
the Alexander von Humboldt Foundation and JSPS KAKENHI Grant Number JP17H01093 are 
gratefully acknowledged.}

	\begin{abstract}
		 We prove sharp two-sided bounds of the fundamental solution for integro-differential operators of order $\alpha \in (0,2)$ that generate a $d$-dimensional Markov 
process. The corresponding Dirichlet form is comparable to that of 
$d$ independent copies of one-dimensional jump processes, i.e., the jumping measure is singular with respect to the $d$-dimensional Lebesgue measure. 
	\end{abstract}
	
	\maketitle
	
\section{Introduction}

Heat kernel bounds play an important role in the study of Markov processes and differential operators. In the theory of partial differential equations, corresponding two-sided  
Gaussian bounds 
are known as Aronson bounds. Given uniformly elliptic coefficients $(a_{ij})$, it is shown in \cite{Aro68} that the fundamental solution $\Gamma(t,x;s,y)$ of the operator $u \mapsto \partial_t u - \partial_i ( a_{ij} \partial_j u )$ satisfies for all $t,s > 0$ and $x,y \in \R^d$ the two-sided estimate
\begin{align}\label{eq:aronson-local}
g_1 (t-s, x-y) \leq \Gamma(t,x;s,y) \leq g_2 (t-s, x-y) \,,
\end{align}
where $g_j(t,x) = a_{j} t^{-d/2} \exp(- b_{j} \frac{|x|^2}{t})$ and $a_j, b_j$ are some positive constants. Up to multiplicative constants,  the fundamental solution of the Laplacian $-\Delta$ bounds the fundamental solution of any uniformly elliptic operator in divergence form from above and below. One main feature of the result is that no further regularity of $(a_{ij})$ as a function on $\R^d$ is required. 
Using a more probabilistic language, estimate \eqref{eq:aronson-local} says that the heat kernel of a non-degenerate diffusion is controlled from above and below by the heat kernel of the Brownian Motion. 

\medskip

A similar result for certain integro-differential operators resp. Markov jump processes is obtained in \cite{BaLe02} and \cite{ChKu03}. Let $K(t,x;s,y)$ denote the fundamental solution of the operator
\begin{align}\label{eq:gen-operator-nonlocal}
u \mapsto \partial_t u - 
\mbox{p.v.} 
\int_{\R^d} \big( u(y) - u(x) \big) J(x,y) \; \d y\,,
\end{align}
where $J(x,y)$ is symmetric and satisfies for some $\alpha \in (0,2)$ and $c_1, c_2 > 0$ the relation $c_1 |x-y|^{-d-\alpha} \leq J(x,y) \leq c_2 |x-y|^{-d-\alpha}$ for all $x \ne y$. Then, analogously to \eqref{eq:aronson-local}, the authors establish for all $t,s > 0$ and $x,y \in \R^d$ the two-sided estimate
\begin{align}\label{eq:aronson-nonlocal}
k_1 (t-s, x-y) \leq K(t,x;s,y) \leq k_2 (t-s, x-y) \,,
\end{align}
where $k_j (t,x) = c_{j} t^{-d/\alpha} \big(1 \wedge 
\frac{t}{|x|^\alpha}\big)^\frac{d+\alpha}{\alpha}$ 
and $c_j$ denotes some positive constant. Note that the functions $k_1, k_2$ are known to be comparable with the heat kernel of the isotropic $\alpha$-stable process. As in the case of a diffusion, it turns out that the heat kernel of a non-degenerate jump process behaves like the heat kernel of the corresponding translation-invariant model process. That is, up to multiplicative constants,  the fundamental solution of the fractional Laplacian $(-\Delta)^{\alpha/2} $ bounds the fundamental solution of corresponding non-degenerate integro-differential operator of the form \eqref{eq:gen-operator-nonlocal}.

\medskip

In other words, pointwise heat kernel bounds are robust under bounded multiplicative changes of the coefficients. This statement can be seen as the result obtained in \cite{Aro68} for the Brownian Motion and confirmed in \cite{BaLe02, ChKu03} for non-degenerate isotropic L\'{e}vy stable processes. The main aim of the present work is to show that the robustness result extends to non-degenerate non-isotropic L\'{e}vy stable processes. 

\medskip

Let us introduce the main objects of our study. Consider a Markov jump process $Z$ in $\R^d$ defined by $Z_t =(Z_t^1, \ldots, Z_t^d)$, where the coordinate processes $Z_t^1, \ldots, Z_t^d$ are independent one-dimensional symmetric stable processes of index $\alpha \in (0, 2)$. The infinitesimal generator of the corresponding semigroup of the process $Z$ is the integro-differential operator 
$L= -(-\partial_{11})^{\alpha/2} -(-\partial_{22})^{\alpha/2} - \ldots  - (-\partial_{dd})^{\alpha/2}$, whose symbol resp. multiplier is given by $\sum\limits_{i=1}^d |\xi^i|^\alpha$. The process $Z$ resp. its generator are not to be mixed up with the isotropic $\alpha$-stable process resp. the fractional Laplace operator $- (-\Delta)^{\alpha/2}$, whose symbol is given by $|\xi|^\alpha$. In this work we show that, up to multiplicative constants,  the fundamental solution of the operator $L$ 
 bounds the fundamental solution of a corresponding non-degenerate integro-differential operator with bounded measurable coefficients. In a more probabilistic fashion: We consider a $d$-dimensional pure jump Markov process $X$ in $\R^d$ whose jump kernel is comparable to that one of the process $Z$. We show that the heat kernels of $Z$ and $X$ satisfy the same sharp two-sided estimates.

\medskip

Let us be more precise. Given $\alpha \in (0,2)$, let $\nu$ be a measure on the Borel sets of $\R^d$ defined by
\begin{align*}
\nu(\d h) = \sum\limits_{i=1}^d |h^i|^{-1-\alpha} \d h^i \prod\limits_{j \ne i} \delta_{\{0\}}(\d h^j).
\end{align*}
Then $\nu$ is a non-degenerate $\alpha$-stable L\'{e}vy measure. Its corresponding  process is the process $(Z_t)$ up to a multiplicative constant. The measure $\nu$ charges only sets that have a nonempty intersection with one of the coordinate axes. For $u \in C^\infty_c(\R^d)$, the corresponding generator $L$ 
can be written as follows   
\begin{align*}
Lu (x) = 
\mbox{p.v.} 
\int_{\R^d} \big(u(x+h) - u(x)\big) \nu(\d h) \qquad (x \in \R^d)
\end{align*} 
and one easily computes for $\xi \in \R^d$
\begin{align*}
\mathcal{F}(-Lu)(\xi) = c_\alpha \big( \sum\limits_{i=1}^d |\xi^i|^\alpha \big) \mathcal{F}(u)(\xi) = c_\alpha \mathcal{F}\big(  (-\partial_{11})^{\alpha/2}u+(-\partial_{22})^{\alpha/2}u+\ldots+ (-\partial_{dd})^{\alpha/2}u \big)(\xi)
\end{align*} 
for some constant depending only on $\alpha$. The corresponding Dirichlet form $(\mathcal{E}^\alpha, D^\alpha)$ on $L^2(\R^d)$ is given by 
\begin{align*}
D^\alpha &=\{u\in L^2(\Rd)| \; \mathcal{E}^\alpha (u,u)<\infty\},\\
\mathcal{E}^\alpha (u,v) &=\int_{\Rd}\Big(\sum_{i=1}^d\int_{\R}\big(u(x+e^i \tau) - u(x)\big)\big(v(x+e^i \tau) -  v(x)\big)  \frac{\d \tau}{|\tau|^{1+\alpha}} \Big) \d x \\
&=\int_{\Rd}\Big(\sum_{i=1}^d\int_{\R}\big(u(x+e^i \tau) - u(x)\big)\big(v(x+e^i \tau) -  v(x)\big) J^\alpha (x,x+e^i \tau) \d \tau \Big) \d x \,,
\end{align*}
where $J^\alpha (x,y) = |y^i-x^i|^{-1-\alpha}$ if, for some $i$,  $x^i \ne y^i$ and $x^j = y^j$ for every $j \ne i$. Note that there is no need to specify all values $J^\alpha (x,y)$ with $(x,y) \in \R^d \times \R^d \setminus \operatorname{diag}$. For simplicity, we set $J^\alpha (x,y) = 0$ if $x^i \ne y^i$ for more than one index $i$.

\medskip

Now we can explain our main result in detail. Let $\alpha \in (0,2)$ and $\Lambda \geq 1$ be given. Assume $J: \R^d \times \R^d \setminus \operatorname{diag}$ is a non-negative function satisfying for all $x \ne y$

\begin{align}\label{eq:J-ellipticity-assum}
\Lambda^{-1} J^\alpha (x,y) \leq J(x,y) \leq \Lambda J^\alpha (x,y).
\end{align}

Set
\begin{align*}
D &=\{u\in L^2(\Rd)| \; \mathcal{E} (u,u)<\infty\},\\
\mathcal{E} (u,v) &=\int_{\Rd}\Big(\sum_{i=1}^d\int_{\R}\big(u(x+e^i \tau) - u(x)\big)\big(v(x+e^i \tau) -  v(x)\big) J (x,x+e^i \tau) \d \tau \Big) \d x \,.
\end{align*}

Let $C_c^1(\Rd)$ be the space of $C^1(\Rd)$ functions with compact support, and $\overline{C_c^1(\Rd)}^{\mathcal{E}_1}$ be the closure of $C_c^1(\Rd)$ in $L^2(\R^d)$ with respect to the metric $(\mathcal{E}_1(\cdot, \cdot))^{1/2}$, where $\mathcal{E}_1(u, u):=\mathcal{E}(u, u)+\|u\|_2^2$. $(\mathcal{E},D)$ is  a regular (symmetric) Dirichlet form on $L^2(\R^d)$ where $D=\overline{C_c^1(\Rd)}^{\mathcal{E}_1}$. Moreover, the corresponding Hunt process $X$ has the H\"{o}lder continuous transition density $p_t(x,y)$ on $(0,\infty)\times \Rd\times\Rd$, see 
\cite{Xu13}.

\medskip
	
Here is the main result of the present work.
\begin{theorem}\label{thm:two-sided_bounds}
There exists $C \geq 1$ such that for any $t>0, x, y\in \R^d$
\begin{align*}
		C^{-1} t^{-d/\alpha}\prod_{i=1}^d\left(1\wedge \frac{t^{1/\alpha}}{|x^i-y^i|}\right)^{1+\alpha}\le 
		p_t( x, y)\le C t^{-d/\alpha}\prod_{i=1}^d\left(1\wedge \frac{t^{1/\alpha}}{|x^i-y^i|}\right)^{1+\alpha}.
\end{align*}
\end{theorem}

The lower bound on $p_t(x, y)$ has already been established in \cite{Xu13} together with some non-optimal upper bound. 

\begin{theorem} \cite[Theorem 3.14 and Theorem 4.21]{Xu13}\label{thm:Xu-upper}
	There exists $C \geq 1$ such that for any $t>0, x, y\in \R^d$
	\begin{align*}
		C^{-1} t^{-d/\alpha}\prod_{i=1}^d\left(1\wedge \frac{t^{1/\alpha}}{|x^i-y^i|}\right)^{1+\alpha}\le 
 p_t( x, y)\le C t^{-d/\alpha}\prod_{i=1}^d\left(1\wedge \frac{t^{1/\alpha}}{|x^i-y^i|}\right)^{\alpha/3}.
	\end{align*}
\end{theorem}

It has been an open problem to establish the matching upper bound. Our result solves this problem and, together with the lower bound 
from \cite{Xu13}, we obtain the desired two-sided heat kernel estimates. 

\medskip

Let us explain the novelty of our paper. 
In general, obtaining the off-diagonal heat kernel
upper bound requires hard work because one has to sum up all the possible trajectories of the process moving from $x$ to $y$ in time $t$.  
For diffusion processes, the so-called Davies method (resp. its extension by Carlen-Kusuoka-Stroock) is a 
very useful analytical method to derive the Gaussian upper 
bound. For jump processes, if the jumping kernel is comparable to a radially symmetric kernel (namely an isotropic case), then the so-called Meyer's decomposition that decomposes the jumping kernel into small jumps and large jumps works well. However, for heat kernel estimates in a non-isotropic setting, there is no useful method known so far.   
In fact, within the framework of 
non-isotropic stable-like processes, the present work is the first one to establish robustness of heat kernel estimates in a non-isotropic setting. Our method is a self-improvement 
method of the off-diagonal upper bound. The idea of this method comes from \cite{BGK09}, however  
\cite{BGK09} treats only the isotropic case and our method 
is much more involved in order to take care of the non-isotropy. 
We note that the proof of the upper bound of \autoref{thm:Xu-upper} by \cite{Xu13} uses the Davies method; it is an interesting question whether one can prove the optimal estimate with this method or not. 

\medskip

Let us formulate a conjecture that, based on the aforementioned results, looks promising. 

\medskip

{\bf Conjecture:} \emph{Let $Z$ be a non-degenerate $\alpha$-stable process in $\mathbb{R}^d$ with L\'{e}vy measure $\nu$. Let $X$ be a symmetric Markov process whose Dirichlet form has a symmetric jump intensity $J(x, d y)$ that is comparable to the one of $Z$, i.e., $J(x, \d y) \asymp \nu(x-\d y)$.  Then the heat kernel of $X$ is comparable to the one of $Z$.}

\medskip 

The conjecture is proved in \cite{ChKu03} in the case $\nu(\d h)= |h|^{-d-\alpha} \d h$. The present work establishes the conjecture in the non-isotropic singular case 
\begin{align*}
\nu(\d h) = \sum\limits_{i=1}^d |h^i|^{-1-\alpha} \d h^i \prod\limits_{j \ne i} \delta_{\{0\}}(\d h^j)
\end{align*}
where $\alpha \in (0,2)$.  Both cases are limit cases for non-degenerate $\alpha$-stable L\'{e}vy measures. Hence, it is plausible to expect the assertion of the conjecture to be true. 

\medskip

We have already mentioned some related results from the literature. Let us mention some further results related to systems of jump processes driven by stable processes resp. to nonlocal operators with singular jump intensities. These works 
mainly address regularity questions, which is a closely related topic. The weak Harnack inequality and H\"{o}lder regularity estimates for solutions to parabolic equations driven by the Dirichlet form $\mathcal{E}$ under assumption \eqref{eq:J-ellipticity-assum} have been established in \cite{KaSc14}. In the elliptic setting, a general approach to the weak Harnack inequality for singular and non-singular cases is developed in \cite{DyKa20}. 

\medskip

Systems of Markov jump processes of the form 
\[
\d Y^i_t = \sum\limits_{j=1}^d A_{ij} (Y_{t-}) \d Z^j_t
\]
with one-dimensional independent symmetric $\alpha_j$-stable components $Z^j_t$, $j=1,\ldots, d$, are studied in several works. In the case 
$\alpha= \alpha_j$ for all $j$,  
\cite{BaCh06} establishes unique weak solvability via the martingale problem. H\"{o}lder regularity of corresponding harmonic functions is provided in \cite{BaCh10}. Of course, some conditions on the matrix-valued function $A$ need to be imposed. In \cite{KRS20} the authors prove the strong Feller property for the corresponding stochastic jump process. They show that for any fixed $\gamma \in (0, \alpha)$ the semigroup $(P_t)$ of the process $(X_t)$ satisfies 
\[
|P_t f(x) - P_tf(y)| \leq c t^{-\frac{\gamma}{\alpha}} |x-y|^\gamma  \|f\|_{\infty}
\]
for bounded Borel functions $f$ and $x,y \in \R^d$. The regularity results of \cite{BaCh10} are extended in \cite{Cha20} to the case where the index of stability $\alpha_j$ is different for different components $Z^j$. Existence and uniqueness for weak solutions in this case is proved in \cite{Cha19}. However, for the uniqueness the article assumes the matrix $A_{ij}$ to be diagonal. Proving uniqueness under natural assumptions in this case seems to be a challenging problem. 
\medskip

In a recent paper \cite{KuRy18}, two-sided heat kernel estimates similar to ours are obtained when  
the matrix $A_{ij}$ is diagonal and the diagonal
elements are bounded and H\"{o}lder continuous. 
They use a method based on Levi's freezing coefficient argument. We note that their method does not seem to work in our case.   
\medskip

Another interesting open question in this context is the Feller property which is established in \cite{KuRy20} under the assumption $\min(\alpha_j) \geq 2/3 \max(\alpha_j)$. This condition is not required in the study of H\"{o}lder regularity of weak solutions  to nonlocal equations stemming from symmetric Dirichlet form as in \cite{KaSc14, DyKa20}, cf. \cite{ChKa20}.

\medskip
	
The paper is organized as follows: In \autoref{sec:prelim} we present the main strategy of our proof and provide some auxiliary results. In \autoref{subsec:strategy} we formulate three lemmas, which imply our main result. \autoref{sec:proof-lemmas} is devoted to the proof of these lemmas. In \autoref{sec:proof-propo} we provide the proof of our main auxiliary result, which is \autoref{prop:main}. Finally, we provide the proof of an important but less innovative auxiliary result (\autoref{lem:symmetry}) in an appendix.

\bigskip
\section{Auxiliary results and strategy of the main proof}\label{sec:prelim}

In this section we present some auxiliary result and discuss the strategy of our proof. In \autoref{subsec:strategy} we explain three lemmas, which are the main building blocks of our proof. In \autoref{sec:auxiliary} we provide a few auxiliary results. First of all, let us explain the notation that we are using. 

\medskip

\emph{Notation:} As is usual, $\N_0$ denotes the non-negative integers including Zero. For two non-negative functions $f$ and $g$, the notation $f\asymp g$ means that there are positive constants $c_1$ and $c_2$ such that $c_1g(x)\leq f (x)\leq c_2 g(x)$ in the common domain of definition for $f$ and $g$. For $a, b\in \R$, we use $a\wedge b$ for $\min \{a, b\}$ and $a\vee b$ for $\max\{a, b\}$. Given any sequence $(a_n)$ of real numbers and $n_1, n_2\in \N_0$, we set $\prod_{n=n_1}^{n_2} a_n$ as equal to $1$ if $n_1>n_2$.

\subsection{Strategy of the main proof}\label{subsec:strategy}

In this subsection, we present the main strategy of the proof of \autoref{thm:two-sided_bounds}. 
As explained in the introduction, due to \cite{Xu13} we only need to establish the upper bound, i.e., we will prove the following result:

\medskip

\begin{theorem}\label{thm:main}
	Assume $\alpha \in(0,2)$. There is a positive constant $C$ such that for all $x, y\in\R^d$, $t>0$ the following estimate holds: 
	\begin{align}\label{eq:upper-bd-mulit}
	p_t(x, y)\le C t^{-d/\alpha}\prod_{i=1}^d\left(\frac{t}{|x^i-y^i|^\alpha}\wedge 1\right)^{1+\alpha^{-1}}\,.
	\end{align}
\end{theorem}

In the remaining part of this section, we explain the skeleton of the proof of \autoref{thm:main}. We are able to reduce the proof to three auxiliary lemmas, which we approach in the following section. For any $q > 0$ and $l\in \{1, \ldots, d-1\}$,  we consider the following conditions.

\medskip

\begin{itemize}
	\item[{\bf $\HHq{q}{0}$}]
	There exists a positive constant 
$C_0=C(q,\Lambda, d, \alpha)$ 
such that for all $t>0,\, x, y\in \R^d$,
	\begin{align}
	p_t(x, y)\le C_0 t^{-d/\alpha}\prod_{i=1}^d\left(\frac{t}{|x^i-y^i|^\alpha}\wedge 
	1\right)^{q}.\label{eq:H0}
	\end{align}
	\item[{\bf $\HHq{q}{l}$}]
	There exists a positive constant $C_l=C(l,q,\Lambda, d, \alpha)$ such that for all $t>0$ and all $x, y\in \R^d$ with $|x^{1}-y^{1}| \leq \ldots \leq |x^{d}-y^{d}|$ the following holds true:
	\begin{align}
			p_t(x, y)\le C_l t^{-d/\alpha}\prod_{i=1}^{d-l}\left(\frac{t}{|x^{i}-y^{i}|^\alpha}\wedge 	1\right)^{q}
	\prod_{i= d-l+1}^d\left(\frac{t}{|x^{i}-y^{i}|^\alpha}\wedge 	1\right)^{1+\alpha^{-1}}\,. \label{eq:Hi}
	\end{align}
\end{itemize}

\begin{remark}\label{rem:constant-Hlq}
	Note that the constant $C_l$ from \eqref{eq:Hi} depends on the jumping kernel $J$ only through the constant $\Lambda$, i.e., different choices of $J$ lead to the same estimate as long as \eqref{eq:J-ellipticity-assum} remains true.
\end{remark}

Let us make some further observations.

\begin{samepage}
	\begin{remark}\label{rem:H-fundamentals}{\ } 
		\begin{enumerate}
			\item{} The assertion of \autoref{thm:main} is equivalent to $\HHq{1+\alpha^{-1}}{l}$ for any $l\in \{0, \ldots, d-1\}$.
			\item{}	Note that $\HHq{q}{l}$ gets stronger 
as $q$ increases, 
that is, for $q<q'$, $\HHq{q'}{l}$ implies $\HHq{q}{l}$.
			\item{} For $q\in [0, 1+\tfrac{1}{\alpha}]$ and $l\le l'$, $\HHq{q}{l'}$ implies $\HHq{q}{l}$.
		\end{enumerate}
	\end{remark}
\end{samepage}

The above three observations can be established easily. The following lemma shows that \eqref{eq:Hi} implies a much stronger result due to \autoref{rem:constant-Hlq}.

\begin{lemma}\label{lem:symmetry}
Assume condition $\HHq{q}{l}$ holds true for some $q\in[0, 1+\alpha^{-1}]$. Then with the same constant 
$C_l=C(l,q,\Lambda,d,\alpha)$ 
for every $t>0$, all $x, y\in \R^d$, and every permutation $\sigma$ of the indices ${1, \ldots, d}$ that satisfies $|x^{\sigma(1)}-y^{\sigma(1)}| \leq \ldots \leq |x^{\sigma(d)}-y^{\sigma(d)}|$ the following holds true:
		\begin{align}
		p_t(x, y) &\le C_l t^{-d/\alpha}\prod_{i=1}^{d-l}\left(\frac{t}{|x^{\sigma(i)}-y^{\sigma(i)}|^\alpha}\wedge 	1\right)^{q}
		\prod_{i= d-l+1}^d\left(\frac{t}{|x^{\sigma(i)}-y^{\sigma(i)}|^\alpha}\wedge 	1\right)^{1+\alpha^{-1}} \label{eq:Hi-permute} \\
		&\le C_l t^{-d/\alpha}\prod_{i=1}^{d-l}\left(\frac{t}{|x^i-y^i|^\alpha}\wedge 	1\right)^{q}
		\prod_{i= d-l+1}^d\left(\frac{t}{|x^i-y^i|^\alpha}\wedge 	1\right)^{1+\alpha^{-1}}. \label{eq:Hi-no-permute}
		\end{align}
\end{lemma}

\begin{remark*}
\autoref{lem:symmetry} reduces the complexity of our iterative argument. In particular, it proves that condition $\HHq{q}{l}$  with $q\in[0, 1+\alpha^{-1}]$ implies condition $\HHq{q}{l}'$  with $q\in[0, 1+\alpha^{-1}]$, which is defined as follows:
\begin{itemize}
	\item [{\bf $\HHq{q}{l}'$}]
	Given $q > 0$ and $l\in \{1, \ldots, d-1\}$ there exists a positive constant 
	$C_l=C(l,q,\Lambda,d,\alpha)$ 
such that for all $t>0$ and all $x, y\in \R^d$ 
	\begin{align} p_t(x, y)\le C_l t^{-d/\alpha}\prod_{i=1}^{d-l}\left(\frac{t}{|x^i-y^i|^\alpha}\wedge 	1\right)^{q}
	\prod_{i= d-l+1}^d\left(\frac{t}{|x^i-y^i|^\alpha}\wedge 	1\right)^{1+\alpha^{-1}}. \label{eq:Hi-super}
\end{align}
\end{itemize}
\end{remark*}

\autoref{lem:symmetry} is trivial if $l=d$ or $l=0$. We provide the proof of \autoref{lem:symmetry} in the appendix.

\medskip

Next, let us explain in detail how the main proof makes use of the condition $\HHq{q}{l}$. Set
\begin{align*}
\lambda_l:=\frac{1}{2}\Big(\sum_{i=1}^{d-l-1}(1+\alpha^{-1})^i\Big)^{-1}\quad \mbox{for} \quad l\in \{0,1,\ldots, d-2\},\quad\mbox{ and }\quad\lambda_{d-1}:=1.
\end{align*}
Note  $\tfrac{1}{2}\big(\tfrac{\alpha}{1+\alpha}\big)^{2(d-l-1)}\le \lambda_l \le \tfrac{\alpha}{2(1+\alpha)}$ for $l\in \{0,1,\ldots, d-2\}$, where the first inequality follows from
\begin{align*}
\lambda_l \geq \tfrac12 \big( (d-l-1) (\tfrac{\alpha + 1}{\alpha})^{d-l-1} \big)^{-1} \geq \tfrac{1}{2}\big(\tfrac{\alpha+1}{\alpha}\big)^{-2(d-l-1)}\,,
\end{align*} 
which makes use of $\tfrac{\alpha+1}{\alpha} > \frac32$. Our main aim is to prove assertion $\HHq{1+\alpha^{-1}}{d-1}$. It will be the last assertion in a sequence of assertions which are proved subsequently in the following order:

\begin{alignat*}{4}
&\HHq{0}{0} &&\hookrightarrow \HHq{\lambda_0}{0} &&\hookrightarrow \HHq{2\lambda_0}{0} \ldots &&\hookrightarrow \HHq{N_0\lambda_0}{0} \\
\hookrightarrow &\HHq{0}{1} &&\hookrightarrow \HHq{\lambda_1}{1} &&\hookrightarrow \,\, \ldots  \ldots \ldots \,\,  &&\hookrightarrow \HHq{N_1\lambda_1}{1} \\
&&&\vdots&&&\vdots&\\
\hookrightarrow &\HHq{0}{d-1} &&\hookrightarrow \HHq{1}{d-1} &&\hookrightarrow \ldots  \ldots \ldots &&\hookrightarrow \HHq{N_{d-1}}{d-1} 
\end{alignat*}
where $N_l:=\lfloor 1+\tfrac{\alpha^{-1}}{\lambda_l} \rfloor$ for $l\in \{0, \ldots, d-1\}$. 
Note $N_{d-2}= \lfloor  1 + \tfrac{2}{\alpha}+\tfrac{2}{\alpha^2} \rfloor \geq  2$
and $N_{0}= \lfloor  1 +  \tfrac{2}{\alpha}\sum_{i=1}^{d-1}(1+\alpha^{-1})^i \rfloor  \geq \tfrac{2}{\alpha}\sum_{i=1}^{d-1}(1+\alpha^{-1})^i \geq (\tfrac{3}{2})^{d-1}$.

\medskip

The above scheme will be established with the help of the following implications. Note that they make use of \autoref{lem:symmetry}.

\begin{lemma}\label{lem:Gl}
Assume condition $\HHq{q}{l}$ holds true for some $l\in\{0, \ldots, d-2\}$ and  $q<\alpha^{-1}$. Then $\HHq{q+\lambda_l}{l}$ holds true.
\end{lemma}

\begin{lemma}\label{lem:Fl}
Assume condition $\HHq{q}{l}$ holds true for some $l\in\{0, \ldots, d-2\}$ and  $q>\alpha^{-1}$. Then $\HHq{0}{l+1}$ holds true.
\end{lemma}

\begin{lemma}\label{lem:GF-dminus1}	{\ } \\
	(i) Assume condition $\HHq{q}{d-1}$ holds true for some $q<\alpha^{-1}$. Then 
	$\HHq{q+1}{d-1}$ holds true. \\
	(ii) Assume condition $\HHq{q}{d-1}$ holds true for some $q>\alpha^{-1}$. Then $\HHq{1+\alpha^{-1}}{d-1}$ holds true.
\end{lemma}

Note that the case $q = \alpha^{-1}$ is left open here. In this case one can apply \autoref{rem:H-fundamentals} (2) and obtain the corresponding conclusion for any $\widetilde{q} < \alpha^{-1}$. Assertion (i) of \autoref{lem:GF-dminus1} and assertion of \autoref{lem:Gl} can be seen as one implication $\HHq{q}{l} \Rightarrow \HHq{q+\lambda_l}{l}$ being true for every $l\in\{0, \ldots, d-1\}$. However, we decide to split the assertion into two cases. As we will see, the proof of \autoref{lem:GF-dminus1} is much simpler than the one of \autoref{lem:Gl}. However, both rely on our main technical result, \autoref{prop:main}. 

\medskip

Altogether we have shown that \autoref{thm:main} follows once we have established \autoref{lem:Gl}, \autoref{lem:Fl} and \autoref{lem:GF-dminus1}.

\subsection{Auxiliary results}\label{sec:auxiliary}

In this subsection we provide several auxiliary results.

\medskip

Let us explain the connection between the kernel $J$ and the corresponding stochastic process $X$. The function $J$ is called the jumping kernel of $X$ and describes the intensity of jumps of the process $X$. The formal relation is given by the following L\'evy system formula, which can be found in \cite[Appendix A]{ChKu08}.
 
\begin{lemma}
For any $x\in \R$, stopping time $S$ (with respect to the filtration of $X$), and non-negative measurable function $f$ on $\R_+ \times \R\times \R$ with $f(s, y, y)=0$ for all $y\in\R$ and $s\ge 0$, we have 
	\begin{equation}\label{eq:LSd}
	\E^x \left[\sum_{s\le S} f(s,X_{s-}, X_s) \right] = \E^x \left[ \int_0^S \left(\sum_{i=1}^{d}\int_{\R} 
	f(s,X_s, X_s+e_ih) 
	J(X_s,X_s+e_ih) dh \right) ds \right].
	\end{equation}
\end{lemma}

Next, let us introduce some results which will be used in the proofs . 
	
\begin{proposition}\cite[Proposition 3.12 (b)]{Xu13}\label{prop:xu1}
	There is a positive constant $C$ such that
		\begin{align*}
		p_t( x, y)\le C t^{-d/\alpha}
		\end{align*}
		for all $t>0$ and $x, y\in \Rd$. Moreover, $p$ is a continuous function in $t,x,y$.
\end{proposition}

For any open set $U\subset \R^d$, 
let $\tau_U:= \inf\{t>0: X_t\notin U\}$ be 
the {\it first exit time} of the process $X$ from $U$.  
	
\begin{proposition}\cite[Proposition 4.4]{Xu13}\label{prop:xu2}
	There is a positive constant $C$ such that
\begin{align*}
\Pp^x(\tau_{B(x, R)}< t)\le C t R^{-\alpha}
\end{align*}
for all $t>0$, $R>0$ and $x\in \Rd$.
\end{proposition}
	
For any non-negative Borel functions $f$ on $\Rd$ and for any $t>0$, $x\in \Rd$, let $\{P_t\}_{\{t\ge 0\}}$ be the transition semigroup of $X$ defined by
$$P_t f(x)=\E^x\left[f(X_t)\right]=\int_{\Rd} 
p_t(x, y)f(y)dy.$$ For any two non-negative measurable functions $f, g$ on $\Rd$, set $$(f, g)=\int _{\Rd}f(x)g(x)dx.$$

	\begin{lemma}\cite[Lemma 2.1]{BGK09}\label{lem:LBGK}
		Let $U$ and $V$ be two disjoint non-empty open subsets of $\Rd$ and $f, g$ be non-negative Borel functions on $\Rd$. 
		Let $\tau=\tau_U$ and $\tau^{'}=\tau_V$ be the first exit times from $U$ and $V$, respectively. Then, for all $a, b, t>0$ such that $a+b=t$, we have
		\begin{align}\label{eq:LBGK}
		\Big(P_tf, g\Big) \le \Big(\E^{x}\big[\1_{\{\tau\le a\}}P_{t-\tau}f(X_{\tau})\big], g\Big)+ \Big(\E^{x}\big[\1_{\{\tau^{'}\le b\}}P_{t-\tau^{'}}g(X_{\tau^{'}})\big], f\Big)\,.
		\end{align}
	\end{lemma}

\bigskip
	
\section{Proof of \autoref{lem:GF-dminus1}, \autoref{lem:Fl} and \autoref{lem:Gl}. }\label{sec:proof-lemmas}

The aim of this section is to provide the proofs of \autoref{lem:GF-dminus1}, \autoref{lem:Fl} and \autoref{lem:Gl}. The proofs rely on an involved technical result, \autoref{prop:main}. In \autoref{subsec:proofs-lemmas} we apply this result and derive the three lemmas. In \autoref{sec:proof-propo} we give the proof of \autoref{prop:main}.

\medskip

Recall that $\alpha\in (0, 2)$ and $t>0$. Given two points $x_0, y_0\in \R^d$ and $t>0$, we need to specify their relative position.

\begin{definition}\label{def:theta_and_R}
	Let $x_0, y_0 \in \R^d$ satisfy $|x_0^{i}-y_0^{i}|\le |x_0^{i+1}-y_0^{i+1}|$ for every $i\in\{1,\ldots, d-1\}$. Let $t>0$, set $\rho:=t^{1/\alpha}$. For $i\in\{1,\ldots, d\}$ define $\lengthR_i\in \Z$ and $R_i > 0$ such that 
	\begin{align}\label{d:thi_Ri}
	\tfrac{5}{4}2^{\lengthR_i} \le  \frac{|x_0^{i}-y_0^{i}|}{\rho}  < \tfrac{10}{4}2^{\lengthR_i} \qquad\mbox{and}\qquad	R_i=2^{\lengthR_i}\rho\,.
	\end{align}
	Then $\lengthR_i\le \lengthR_{i+1}$ and $R_i\le R_{i+1}$. We say that condition $\mathcal{R}(i_0)$ holds if 
	\begin{align}\label{eq:case-R-i0}
	\lengthR_1\le \ldots \le \lengthR_{i_0-1}\le 0<1\le \lengthR_{i_0}\le \ldots\le \lengthR_d\, 
	\tag*{$\mathcal{R}(i_0)$} .
	\end{align}
	We say that condition $\mathcal{R}(d+1)$ holds if $\lengthR_1\le \ldots \le \lengthR_d \leq 0 < 1$.
\end{definition}

In the proof of $\HHq{q}{0}$ and $\HHq{q}{l}$ we need to consider arbitrary tuples $(x_0, y_0)$. In specific geometric situations, the implications  \autoref{lem:GF-dminus1}, \autoref{lem:Fl} and \autoref{lem:Gl} follow directly.

\begin{lemma}\label{lem:reduction_cases}
Let $t > 0$ and $x_0, y_0$ be two points in $\R^d$ satisfying condition \ref{eq:case-R-i0} for $i_0 \in \{d-l+1, \ldots, d+1\}$.  Assume that $\HHq{q}{l}$ holds for some $l \in \{0,\ldots, d-1\}$ and $q \geq 0$. Then 
\begin{align}\label{eq:reduction_cases}
p_t(x_0, y_0) \le C t^{-d/\alpha}\prod_{i=1}^{d} \left(\frac{t}{|x_0^{i}-y_0^{i}|^\alpha}\wedge 	1\right)^{1+\alpha^{-1}}
\end{align}
for some constant $C > 0$ independent of $t$ and $x_0, y_0$.
\end{lemma}
\begin{proof}
(i) The case $\mathcal{R}(d+1)$ (i.e. $|x_0^{d}-y_0^{d}|< \tfrac{10}{4} t^{1/\alpha} $) is simple because, in this case \autoref{prop:xu1} implies $p_t(x_0, y_0) \le c t^{-d/\alpha}$. Estimate \eqref{eq:reduction_cases} follows.  
	
(ii) Assume, condition \ref{eq:case-R-i0} holds true for some $i_0 \in \{d-l+1, \ldots, d\}$. Then for any $j\le d-l\le i_0-1$
	\begin{align*}
	|x_0^{j}-y_0^{j}|<\tfrac{10}{4}2^{\lengthR_j}\rho \le \tfrac{10}{4}2^{\lengthR_{i_0-1}}\rho\le \tfrac{10}{4}\rho \quad \Longrightarrow \quad 
	\Big(\tfrac{t}{|x_0^{j}-y_0^{j}|^\alpha}\wedge 1\Big)\asymp 1 \,.
	\end{align*}
	Hence, $\HHq{q}{l}$ implies 
	\begin{align*}
	p_t(x_0, y_0) &\le C_l t^{-d/\alpha}\prod_{i=1}^{d-l}\left(\frac{t}{|x_0^{i}-y_0^{i}|^\alpha}\wedge 	1\right)^{q} \prod_{i= d-l+1}^d\left(\frac{t}{|x_0^{i}-y_0^{i}|^\alpha}\wedge 	1\right)^{1+\alpha^{-1}} \\
	&\asymp t^{-d/\alpha}\prod_{i=1}^{d} \left(\frac{t}{|x_0^{i}-y_0^{i}|^\alpha}\wedge 	1\right)^{1+\alpha^{-1}}\,.
	\end{align*}		
\end{proof}

\begin{remark*}
\autoref{lem:reduction_cases} allows us in the arguments below to restrict ourselves to the case \ref{eq:case-R-i0} for $i_0\in \{1,\ldots, d-l\}$. 
\end{remark*}

Here is our main technical result. 

\begin{proposition}\label{prop:main}
Let $\alpha \in (0,2)$. Assume that  $\HHq{q}{l}$ holds true for some $l\in \{0,1,\ldots, d-1\}$ and  $q \in [0,1 + \alpha^{-1}]$.
Let $t > 0$, set $\rho=t^{1/\alpha}$. Consider  $x_0, y_0\in \Rd$ satisfying the condition \ref{eq:case-R-i0} for some $i_0 \in \{1, \dots, d-l\}$. Let $f$ be a non-negative Borel function on $\Rd$ supported in $B(y_0, \tfrac{\rho}{8})$.
For $j_0 \in \{i_0,\ldots, d-l\}$ define an exit time $\tau$ by $\tau=\tau_{B(x_0, {R_{j_0}}/{8})}$, where $R_j=2^{\theta_j}\rho$ as in \eqref{d:thi_Ri}. Then there exists $C_{\ref{prop:main}}>0$ independent of  $x_0, y_0$ and $t$ such that for every $x\in B(x_0, \frac{\rho}{8})$,
	\begin{align}\label{eq:main1}
	\begin{split}
	&\E^{x}\left[\1_{\{\tau\le t/2\}}P_{t-\tau}f(X_{\tau})\right]\\
	&\le \,C_{\ref{prop:main}} t^{-d/\alpha} \|f\|_1  \prod_{j=j_0+1}^{d-l}
	\left(\tfrac{t}{|x_0^{j}-y_0^{j}|^\alpha} \wedge 1 \right)^q\prod_{j=d-l+1}^d\left(\tfrac{t}{|x_0^{j}-y_0^{j}|^{\alpha}} \wedge 1 \right)^{1+\alpha^{-1}} \\
	& \qquad \times 
	\begin{cases}
	\left(\tfrac{t}{| x_0^{j_0}-y_0^{j_0}|^\alpha} \wedge 1 \right)^{1+q}
	&\mbox{ if } q<\alpha^{-1}\\
	\left(\tfrac{t}{|x_0^{j_0}-y_0^{j_0}|^\alpha} \wedge 1 \right)^{1+\alpha^{-1}}
	&\mbox{ if } q>\alpha^{-1}.
	\end{cases}
	\end{split}
	\end{align}
\end{proposition}

We postpone the proof of \autoref{prop:main} until \autoref{sec:proof-propo}.

\subsection{Proof of \autoref{lem:GF-dminus1}, \autoref{lem:Fl} and \autoref{lem:Gl}} \label{subsec:proofs-lemmas} {\ }\\

In this subsection we explain how to apply \autoref{prop:main} and derive \autoref{lem:GF-dminus1}, \autoref{lem:Fl} and \autoref{lem:Gl}. This proves \autoref{thm:main}.

Before we provide the actual proofs, let us explain how the estimate \eqref{eq:main1} is used in our approach. Consider non-negative Borel functions $f, g$ on $\R^d$ supported in $B(y_0, \frac{\rho}{8})$ and $B(x_0, \frac{\rho}{8})$, respectively.  
We apply 
\autoref{lem:LBGK} with functions $f, g$,  subsets $U:=B(x_0, s), V:=B(y_0, s)$ for some $s>0$, $a=b=t/2$ and $\tau=\tau_{U}, \tau^{'}=\tau_{V}$.
The first term of the right hand side of \eqref{eq:LBGK} is
\begin{align}\label{eq:m1}
\left(\E^{\cdot}\left[\1_{\{\tau\le t/2\}}P_{t-\tau}f(X_{\tau})\right], g\right)=\int_{B(x_0, \frac{\rho}{8})}\E^{x}\left[\1_{\{\tau\le t/2\}}P_{t-\tau}f(X_{\tau})\right]\,g(x)dx ,
\end{align}
and a similar identity holds for the second term. These identities are used in the proofs below.

\medskip

\begin{proof}[Proof of \autoref{lem:GF-dminus1}]
As mentioned in \autoref{rem:H-fundamentals}, $\HHq{q}{d-1}$ implies  $\HHq{1+\alpha^{-1}}{d-1}$ for $q> 1+\alpha^{-1}$. Thus when proving (ii) we may limit ourselves to the case $q\in (\frac{1}{\alpha}, 1+\frac{1}{\alpha}]$ for the rest of the proof.
Hence, for (ii) 
we assume $\HHq{q}{d-1}$ for $q\in (\frac{1}{\alpha}, 1+\frac{1}{\alpha}]$. Because of \autoref{lem:reduction_cases}, for any $t>0$ we only need to consider the case where $x_0, y_0$ satisfy condition $\mathcal{R}(1)$. 
Recall that $\rho=t^{1/\alpha}$ and $R_1=2^{\lengthR_1}\rho$. Applying \autoref{prop:main} with $i_0=j_0=1$, 
for $x\in B(x_0, \frac{\rho}{8})$ and $\tau=\tau_{B(x_0, \frac{R_{1}}{8})}$, we obtain
\begin{align*}
\E^{x}\left[\1_{\{\tau\le t/2\}}P_{t-\tau}f(X_{\tau})\right]
\le C_{\ref{prop:main}}t^{-d/\alpha} \|f\|_1
\prod_{j=2}^d\left(\tfrac{t}{|x_0^j-y_0^j|^{\alpha}}\right)^{1+\alpha^{-1}}
\begin{cases}
\left(\tfrac{t}{|x_0^1-y_0^1|^\alpha}\right)^{1+q}
&\mbox{ if } q<\alpha^{-1},\\
\left(\tfrac{t}{|x_0^1-y_0^1|^\alpha}\right)^{1+\alpha^{-1}}
&\mbox{ if } q>\alpha^{-1},
\end{cases}
\end{align*}
and by \eqref{eq:m1}, 
\begin{align*}
&\left(\E^{\cdot}\left[\1_{\{\tau\le t/2\}}P_{t-\tau}f(X_{\tau})\right], g\right)\nn\\
&\le  ct^{-d/\alpha} \|f\|_1\|g\|_1
\prod_{j=2}^d\left(\tfrac{t}{|x_0^j-y_0^j|^{\alpha}}\right)^{1+\alpha^{-1}}
\begin{cases}
\left(\tfrac{t}{|x_0^1-y_0^1|^\alpha}\right)^{1+q}
&\mbox{ if } q<\alpha^{-1},\\
\left(\tfrac{t}{|x_0^1-y_0^1|^\alpha}\right)^{1+\alpha^{-1}}
&\mbox{ if } q>\alpha^{-1}.
\end{cases}
\end{align*}
Similarly we obtain the second term of right hand side of \eqref{eq:LBGK} and therefore,
\begin{align*}
\left(P_tf, g\right)\le ct^{-d/\alpha} \|f\|_1\|g\|_1
\prod_{j=2}^d\left(\tfrac{t}{|x_0^j-y_0^j|^{\alpha}}\right)^{1+\alpha^{-1}}
\begin{cases}
\left(\tfrac{t}{|x_0^1-y_0^1|^\alpha}\right)^{1+q}
&\mbox{ if } q<\alpha^{-1},\\
\left(\tfrac{t}{|x_0^1-y_0^1|^\alpha}\right)^{1+\alpha^{-1}}
&\mbox{ if } q>\alpha^{-1}.
\end{cases}
\end{align*} 

Since $P_tf(x)=\int_{\Rd} p_t(x, y)f(y)dy$ and $p$ is a continuous function, we obtain for $t>0$ 
and $x_0, y_0$ satisfying the condition $\mathcal{R}(1)$ the following estimate
 \begin{align}
\label{eq:con1}
 p_t(x_0,y_0)
& \le c t^{-d/\alpha}\prod_{j=2}^d\left(\tfrac{t}{|x_0^j-y_0^j|^{\alpha}}\right)^{1+\alpha^{-1}}
 \begin{cases}
 \left(\tfrac{t}{|x_0^1-y_0^1|^\alpha}\right)^{1+q}
 &\mbox{ if } q<\alpha^{-1},\\
 \left(\tfrac{t}{|x_0^1-y_0^1|^\alpha}\right)^{1+\alpha^{-1}}
 &\mbox{ if } q>\alpha^{-1}
 \end{cases}\nn\\
 &\asymp c t^{-d/\alpha}\prod_{j=2}^d\left(\tfrac{t}{|x_0^j-y_0^j|^{\alpha}}\wedge 1\right)^{1+\alpha^{-1}}
 \begin{cases}
 \left(\tfrac{t}{|x_0^1-y_0^1|^\alpha}\wedge 1\right)^{1+q}
 &\mbox{ if } q<\alpha^{-1},\\
 \left(\tfrac{t}{|x_0^1-y_0^1|^\alpha}\wedge 1\right)^{1+\alpha^{-1}}
 &\mbox{ if } q>\alpha^{-1}.
 \end{cases}
 \end{align}
This proves  \autoref{lem:GF-dminus1}.
\end{proof}

\medskip
\begin{proof}[Proof of \autoref{lem:Fl}]
Let $l\in \{0, 1, \ldots, d-2\}$, $t>0$ and $x_0, y_0$ satisfy the condition \ref{eq:case-R-i0} for some $i_0\in\{1,\ldots, d-l\}$. As noted in \autoref{rem:H-fundamentals} $\HHq{q}{l}$ implies  $\HHq{1+\alpha^{-1}}{l}$ for $q> 1+\alpha^{-1}$.
Thus, we limit ourselves to the case 
$q\in (\alpha^{-1}, 1+\alpha^{-1}]$ 
for the rest of the proof.
Assume {$\HHq{q}{l}$} 
for $q\in (\alpha^{-1}, 1+\alpha^{-1}]$. Recall that $\rho=t^{1/\alpha}$ and $R_i=2^{\lengthR_i}\rho$. By \autoref{prop:main} with $j_0=d-l$, 
for $x\in B(x_0, \frac{\rho}{8})$ and $\tau=\tau_{B(x_0,\frac{R_{d-l}}{8})}$, we obtain 
\begin{align*}
\E^{x} &\left[\1_{\{\tau\le t/2\}}P_{t-\tau}f(X_{\tau})\right]\le \, C_{\ref{prop:main}}\,t^{-d/\alpha} \|f\|_1 \prod_{j=d-l}^d\left(\tfrac{t}{|x_0^j-y_0^j|^{\alpha}}\right)^{1+\alpha^{-1}}.
\end{align*}
Similarly to the proof of \autoref{lem:GF-dminus1}, we obtain for $t>0$ and for a.e. $(x, y)\in B(x_0, \frac{\rho}{8})\times B(y_0, \frac{\rho}{8})$, 
\begin{align*}
p_t(x, y)\le c t^{-d/\alpha}
\prod_{j=d-l}^d\left(\tfrac{t}{|x_0^j-y_0^j|^{\alpha}}\right)^{1+\alpha^{-1}}\,,
\end{align*}
and therefore for $t>0$ and $x_0, y_0$ satisfying the condition \ref{eq:case-R-i0},
\begin{align}\label{eq:con2}
p_t(x_0, y_0)&\le c t^{-d/\alpha}\prod_{i=d-l}^{d}\left(\frac{t}{|x_0^i-y_0^i|^{\alpha}}\wedge 1\right)^{1+\alpha^{-1}}.
\end{align}
This implies $\HHq{0}{l+1}$ for $l\in \{0, 1, \ldots, d-2\}$  by \autoref{lem:reduction_cases} and hence we have proved \autoref{lem:Fl}.
\end{proof}

\medskip

The proof of \autoref{lem:Gl} is more complicated.

\medskip
\begin{proof}[Proof of \autoref{lem:Gl}]
Let $l\in \{0, 1, \ldots, d-2\}$, $t>0$ and $x_0, y_0$ satisfy the condition \ref{eq:case-R-i0} for some $i_0\in\{1,\ldots, d-l\}$. Assume $\HHq{q}{l}$ for $ 0\le q<{\alpha}^{-1}$. Recall that $\rho=t^{1/\alpha}$ and $R_i=2^{\lengthR_i}\rho$.
By \autoref{prop:main} with  $j_0\in \{i_0, \ldots, d-l\}$, we obtain that for $x\in B(x_0, \frac{\rho}{8})$ and $\tau=\tau_{B(x_0, \frac{R_{j_0}}{8})}$
\begin{align*}
\E^{x}\left[\1_{\{\tau\le t/2\}}P_{t-\tau}f(X_{\tau})\right]
\le C_{\ref{prop:main}}t^{-d/\alpha} \|f\|_1
G_{j_0}(l)
\end{align*}
where
$$G_{j_0}(l):=\left(\tfrac{t}{|x_0^{j_0}-y_0^{j_0}|^\alpha}\right)^{1+q}\prod_{j={j_0+1}}^{d-l}
\left(\tfrac{t}{|x_0^{j}-y_0^{j}|^\alpha}\right)^q\prod_{j=d-l+1}^d\left(\tfrac{t}{|x_0^{j}-y_0^{j}|^{\alpha}}\right)^{1+\alpha^{-1}}.$$
Similarly to the proof of 
\autoref{lem:GF-dminus1}, we obtain for $t>0$ and for a.e. $(x, y)\in B(x_0, \frac{\rho}{8})\times B(y_0, \frac{\rho}{8})$
\begin{align*}
p_t(x, y)\le c t^{-d/\alpha}
G_{j_0}(l)\qquad\mbox{ for } j_0\in \{i_0, \ldots, d-l\},
\end{align*}
and hence for $t>0$ and $x_0, y_0$ satisfying the condition \ref{eq:case-R-i0},
\begin{align}\label{eq:upp11}
p_t(x_0, y_0)&\le c t^{-d/\alpha}\Big(G_{i_0}(l)\wedge G_{i_0+1}(l)\wedge \ldots\wedge G_{d-l}(l)\Big).
\end{align}

Given $l\in \{0, 1, \ldots, d-2\}$, in order to obtain $\lambda_l$ in \autoref{lem:Gl}, we first define $\lambda_l^{i_0}$ inductively for $i_0\in\{1, \ldots, d-l\}$. 
First, let 
$i_0=d-l$ and $j_0=d-l$.  By \eqref{eq:upp11}, for $t>0$ and $x_0, y_0$ satisfying $\mathcal{R}(d-l)$, we obtain that
\begin{align}\label{eq:d-l}
p_t(x_0, y_0)&\le c t^{-d/\alpha} \left(\tfrac{t}{|x_0^{d-l}-y_0^{d-l}|^\alpha}\right)^{1+q}\prod_{j=d-l+1}^d\left(\tfrac{t}{|x_0^{j}-y_0^{j}|^{\alpha}}\right)^{1+\alpha^{-1}}\nn\\
&\asymp c t^{-d/\alpha}\prod_{j={1}}^{d-l}
\left(\tfrac{t}{|x_0^{j}-y_0^{j}|^\alpha}\wedge 1\right)^{\lambda_l^{d-l}+q}\prod_{j=d-l+1}^d\left(\tfrac{t}{|x_0^{j}-y_0^{j}|^{\alpha}}\wedge 1\right)^{1+\alpha^{-1}}.
\end{align}
where $\lambda_l^{d-l}:=1$. 
If $i_0\in\{1, \ldots, d-l-1\}$, we set $$\lambda_l^{i_0}:=\frac{1}{2}\Big(\sum_{i=1}^{d-l-i_0}(1+\alpha^{-1})^{i}\Big)^{-1}\le \frac{1}{2}.$$
For $a, b\in \N_0$, define $\TT{a}{b}:=\Tj{a}{b}$ if $a\le b$, and otherwise  $\TT{a}{b}:=0$. If $0<\lambda_l^{i_0}\le \frac{\lengthR_{j_0}-q\,\TT{{i_0}}{j_0-1}}{\TT{{i_0}}{d-l}}$
for some  $j_0\in\{{i_0},\ldots, d-l-1\}$, 
since $\tfrac{t}{R_i^\alpha}= 2^{-\alpha \lengthR_i}$, we obtain  

\begin{align}
\Big(\tfrac{t}{R_{j_0}^\alpha}\Big)^{1+q} \prod_{j=j_0+1}^{d-l}\Big(\tfrac{t}{R_j^{\alpha}}\Big)^{q}
&= 2^{-\lengthR_{j_0}\alpha(1+q)} 2^{-\TT{j_0+1}{d-l}\alpha q} \nonumber \\ &\le 2^{-\alpha\TT{{i_0}}{d-l}(q+\lambda_l^{i_0})}=\prod_{j={i_0}}^{d-l}\Big(\tfrac{t}{R_j^{\alpha}}\Big)^{q+\lambda_l^{i_0}}\,. \label{eq:psa1}
\end{align}

On the other hand, if $\lambda_l^{i_0}>\max_{j_0\in\{{i_0}, \ldots, d-l-1\}}\left(\frac{\lengthR_{j_0}-q\,\TT{{i_0}}{j_0-1}}{\TT{{i_0}}{d-l}}\right)$, 
since
\begin{alignat*}{3}
&&&&\lambda_l^{i_0} &> \frac{\lengthR_{i_0}}{\TT{i_0}{d-l}}\,,\\
\lambda_l^{i_0} &> \frac{\lengthR_{i_0+1}}{\TT{i_0}{d-l}}-\lambda_l^{i_0} q \,\,
&&\Longrightarrow \,\, &(1+q)\lambda_l^{i_0} &\ge \frac{\lengthR_{i_0+1}}{\TT{i_0}{d-l}}\,,\\
\lambda_l^{i_0} &> \frac{\lengthR_{i_0+2}}{\TT{i_0}{d-l}}-(1+(1+q))\lambda_l^{i_0} q\,\,
&&\Longrightarrow\,\, &(1+q)^2\lambda_l^{i_0}&\ge \frac{\lengthR_{i_0+2}}{\TT{i_0}{d-l}}\,,\\
&\vdots &&\vdots &&\vdots \\
\lambda_l^{i_0} &> \frac{\lengthR_{d-l-1}}{\TT{i_0}{d-l}}-\Big(\sum_{i=1}^{d-l-1-i_0}(1+q)^{i-1}
\Big)\lambda_l^{i_0} q \quad 
&&\Longrightarrow \quad &(1+q)^{d-l-1-i_0}\lambda_l^{i_0} &\ge \frac{\lengthR_{d-l-1}}{\TT{i_0}{d-l}}\,,
\end{alignat*}
for $q<{\alpha}^{-1}$, we have that 
\begin{align*}
\frac{\TT{i_0}{d-l-1}}{\TT{i_0}{d-l}}(1+q)\le  \lambda_l^{i_0}\sum_{i=1}^{d-l-i_0}(1+q)^i\le \lambda_l^{i_0}\sum_{i=1}^{d-l-i_0}(1+\alpha^{-1})^i\le \frac{1}{2}.
\end{align*}
These observations imply that 
\begin{align*}
\frac{\lengthR_{d-l}}{\TT{i_0}{d-l}}(1+q)-q=1-\frac{\TT{i_0}{d-l-1}}{\TT{i_0}{d-l}}(1+q)\ge  \frac{1}{2}\ge \lambda_l^{i_0},
\end{align*}
and hence
\begin{align}\label{eq:psad1} 
\left(\tfrac{t}{R_{d-l}^{\alpha}}\right)^{1+q}= 2^{-\alpha \lengthR_{d-l}(1+q)}\le  2^{-\alpha\big(\TT{i_0}{d-l}\big)(q+\lambda_l^{i_0})}
=\prod_{j=i_0}^{d-l}\Big(\tfrac{t}{R_j^{\alpha}}\Big)^{q+\lambda_l^{i_0}}.
\end{align}
Since $R_i\asymp |x_0^i-y_0^i|$ (see \eqref{d:thi_Ri}), \eqref{eq:psa1}--\eqref{eq:psad1} yield that 
\begin{align*}
\Big(G_{i_0}(l)\wedge G_{i_0+1}(l)\wedge \ldots\wedge G_{d-l}(l)\Big)
&\le c \prod_{j=i_0}^{d-l}\Big(\tfrac{t}{|x_0^j-y_0^j|^{\alpha}}\Big)^{q+\lambda_l^{i_0}}
\prod_{i=d-l+1}^{d}\left(\tfrac{t}{|x_0^j-y_0^j|^{\alpha}}\right)^{1+\alpha^{-1}}.
\end{align*}
Combining the above inequality with \eqref{eq:upp11}, for any $t>0$ and $x_0, y_0$ satisfying \ref{eq:case-R-i0}, $i_0\in\{1, \ldots, d-l-1\}$, we obtain 
\begin{align}\label{eq:others}
p_t(x_0,y_0)& \le\,c t^{-d/\alpha}\prod_{j=i_0}^{d-l}\Big(\tfrac{t}{|x_0^j-y_0^j|^{\alpha}}  \Big)^{q+\lambda_l^{i_0}}\prod_{i=d-l+1}^{d}\left(\tfrac{t}{|x_0^j-y_0^j|^{\alpha}}  \right)^{1+\alpha^{-1}}\nn\\
&\asymp  t^{-d/\alpha}\prod_{i=1}^{d-l}\left(\tfrac{t}{|x_0^i-y_0^i|^{\alpha}}\wedge 1\right)^{q+\lambda_l^{i_0}}\prod_{i=d-l+1}^{d}\left(\tfrac{t}{|x_0^i-y_0^i|^{\alpha}}\wedge 1\right)^{1+\alpha^{-1}},
\end{align}
where $\lambda_l^{i_0}=\frac{1}{2}\Big(\sum_{i=1}^{d-l-i_0}(1+\alpha^{-1})^{i}\Big)^{-1}.$
Therefore by \eqref{eq:d-l} and \eqref{eq:others} in connection with \autoref{lem:reduction_cases}, 
we have $\HHq{q+\lambda_l}{l}$ with $\lambda_l:=\min_{i_0\in \{1,2,\ldots,d-l\}}\lambda_l^{i_0}=\frac{1}{2}\Big(\sum_{i=1}^{d-l-1}(1+\alpha^{-1})^{i}\Big)^{-1}$ for $l\in \{0, 1, \ldots, d-2\}$. Hence the proof of \autoref{lem:Gl} is complete.
\end{proof}

\medskip
	
Using \autoref{prop:main} we have established  \autoref{lem:GF-dminus1}, \autoref{lem:Fl} and \autoref{lem:Gl}. This completes the proof of \autoref{thm:main}.

\bigskip

\section{Proof of \autoref{prop:main}} \label{sec:proof-propo}

In this section, we present the proof of our main technical result, \autoref{prop:main}. The proof is based on several auxiliary observations and computations. 

\medskip

We first introduce a decomposition of $\R^d$ given by sets $D_k \subset \R^d$. Later, in the main proof we 
fix $y_0 \in \R^d$ and $\rho > 0$, and 
work with sets $A_k = y_0+\rho D_k$.

\begin{definition}\label{def:D-sets}{\ }
	\begin{enumerate}
		\item[(0)] Define $D_0 =  \bigcup_{i=1,\ldots, d} \{ |x_i| < 1\} \cup (-2,2)^d $.
		\item[(i)] Given $k\in \N, \gamma \in \N_0^d$ with $\sum_{i=1}^d \gamma_i = k$ and $\switch \in \{-1,1\}^d \,$,
		define a box (hyper-rectangle) $D_k^{\gamma, \switch}$ by
		\begin{align*}
		D_k^{\gamma, \switch} = \switch_1 [2^{\gamma_1},2^{\gamma_1+1}) \times 
		\switch_2 [2^{\gamma_2},2^{\gamma_2+1}) \times \ldots \times 
		\switch_d [2^{\gamma_d},2^{\gamma_d+1}) \,. 
		\end{align*}
		\item[(ii)] Given $k\in \N$ and $\gamma \in \N_0^d$ with $\sum_{i=1}^d \gamma_i = k$, 
		define
		\begin{align*}
		D_k^{\gamma} = 
		\bigcup_{\switch \in \{-1,1\}^d}  D_k^{\gamma, \switch}	\,. 
		\end{align*}
		\item[(iii)] Given $k\in \N$, define
		\begin{align*}
		D_k = \bigcup_{\gamma \in \N_0^d:\, \sum_{i=1}^d \gamma_i = k}  D_k^{\gamma}  \,. 
	\end{align*}
	\end{enumerate}
\end{definition}

Next, we define shifted boxes with center $y_0$. 
For $y_0\in \R^d$, $t > 0$ and $\rho=t^{1/\alpha}$ 
let $A_0:=y_0+\rho D_0$. Given $k \in \N,  \gamma \in \N_0^d$ with $\sum_{i=1}^d \gamma_i = k$ and $\switch \in \{-1,1\}^d$, we define 
\begin{align*}
A_{k, \gamma, \switch} &=y_0+\rho D_k^{\gamma, \switch} \,, \\
A_{k, \gamma} &= y_0+\rho D_k^{\gamma} \,, \\
A_k &= y_0+\rho D_k \,.
\end{align*}

\medskip

\begin{figure}
	\includegraphics{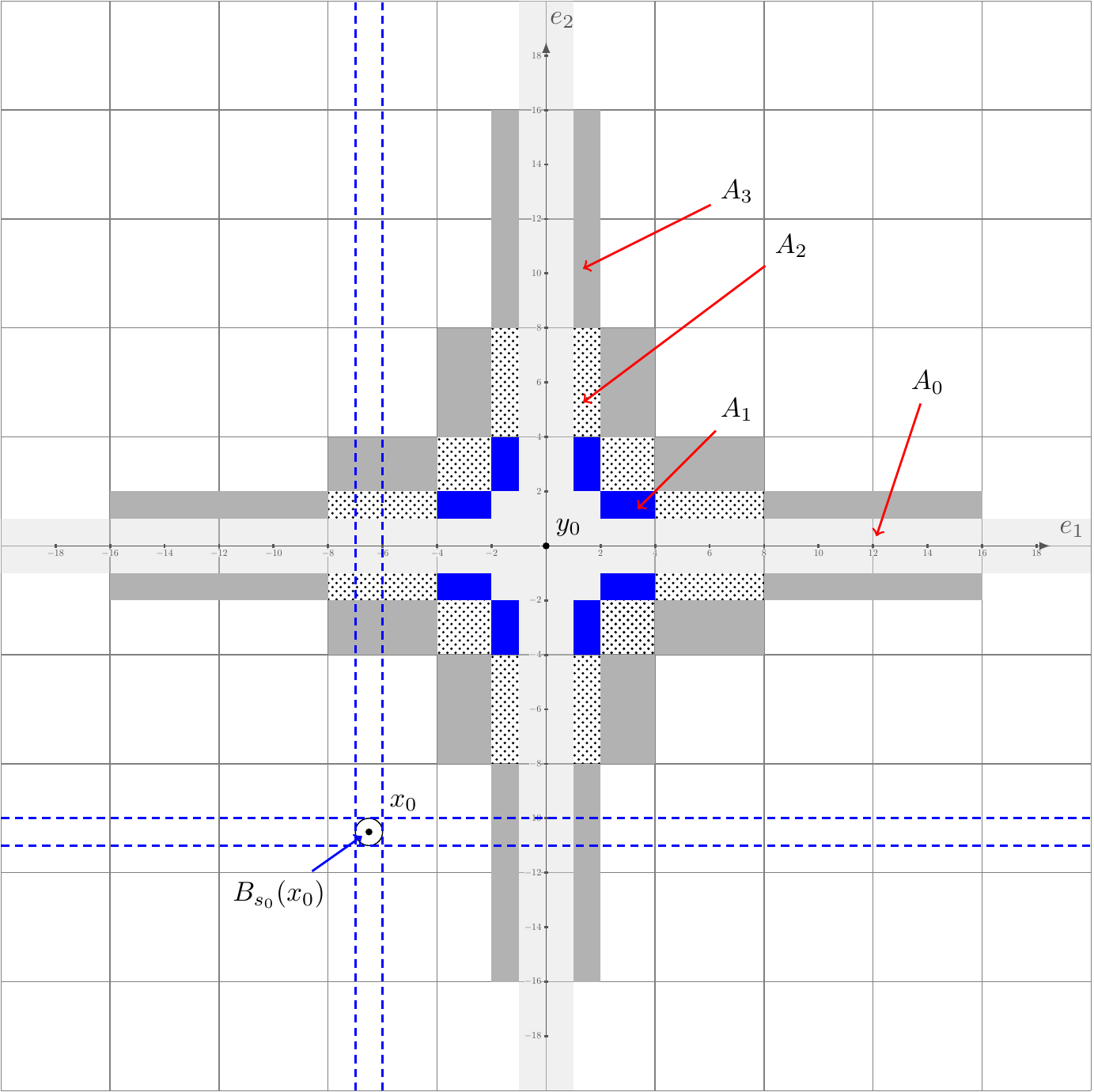}
	\caption{The sets $A_k$ and $B_{S_0}(x_0):=B(x_0, s(j_0))$}
	\label{fig:1}
\end{figure}

\medskip

Let us collect some useful properties of the boxes resp. the corresponding decomposition. We formulate the results for the sets $D_k$ but they imply corresponding results for the sets $A_k$ such as $A_{k}\cap A_{l}=\emptyset$ for $k\neq l$ and $\cup_{k=0}^{\infty}A_k =\R^d$.

\begin{lemma}
Let $k \in \N_0,  \gamma \in \N_0^d$ and $\switch \in \{-1,1\}^d \,$.
	\begin{enumerate}
		\item Given $k\in \N, \gamma$ with $\sum_{i=1}^d \gamma_i = k$, there are $2^d$ sets of the form $D_k^{\gamma, \switch}$, 
		and $|D_k^{\gamma, \switch}|= \prod_{i=1}^d 2^{\gamma_i} = 2^{k}$. 
		\item Given $k \in \N, \switch \in \{-1,1\}^d$, there are 
		$\left( \begin{array}{c} d+k-1 \\ d-1 \end{array} \right)$ sets $D_k^{\gamma, \switch}$ with $\sum_{i=1}^d \gamma_i = k$. Thus, the set $D_k$ consist of $2^d \left(\begin{array}{c} d+k-1 \\ d-1 \end{array} \right)$ disjoint boxes. 
		\item $D_k \cap D_l = \emptyset$ if $k \ne l$ and 
		$\bigcup_{k \in \N_0} D_k = \R^d$.
	\end{enumerate}
\end{lemma}

Since the proof is elementary, we omit it.  

\medskip

The proof of \autoref{prop:main} is based on several technical observations. We assume $l\in \{0,1,\ldots, d-1\}$ and $i_0 \in \{1, \dots, d-l\}$. We assume $t>0$ and $x_0, y_0 \in \R^d$ such that the condition \ref{eq:case-R-i0} is satisfied. We set $\rho = t^{1/\alpha}$. The main idea is to use a decomposition of the left-hand side in \eqref{eq:main1} according to the following: 
\begin{align}\label{eq:dec}
\E^{x}\left[\1_{\{\tau\le t/2\}}P_{t-\tau}f(X_{\tau})\right]=\sum_{k=0}^\infty \E^{x}\left[\1_{\{\tau\le t/2\}}\1_{\{X_{\tau}\in A_k\}}P_{t-\tau}f(X_{\tau})\right] \,,
\end{align}
where $x\in B(x_0, \frac{\rho}{8})$, $s(j_0)=\frac{R_{j_0}}{8}$ and $\tau=\tau_{B(x_0, s(j_0))}$. This approach requires a careful tracking of the random points $X_\tau$. We will need some refinements of the decomposition \eqref{eq:dec}. To this end, set
\begin{alignat}{2}
\Phi(k)&:=\E^{x}\left[\1_{\{\tau\le t/2\}}\1_{\{X_{\tau}\in A_k\}}P_{t-\tau}f(X_{\tau})\right] \qquad &&\text{ for } 
k\in \N_0, \label{eq:def-Phi-k}\\
\Phi^i(k)&:=\E^{x}\left[\1_{\{\tau\le t/2\}}\1_{\{X_{\tau}\in A_k^{e_i}\}}P_{t-\tau}f(X_{\tau})\right] \qquad &&\text{ for } k\in \N_0,\,\, i\in \{1,2,\ldots, d\}\,, \label{eq:def-Phi-i-k}
\end{alignat}
where, for $k \in \N_0$ and $i \in \{1,2, \ldots, d\}$, we set
\begin{align*}
A_k^{e_i}: = A_k \cap  \bigcup\limits_{u \in B(x_0,s(j_0))} \{ u +  h e_i | \,  h \in \R\} \,.
\end{align*}
The set $A_k^{e_i}$ contains all possible points in $A_k$ that the process $X$ can jump to when it leaves the ball $B(x_0,s(j_0))$ by a jump in direction $i$. Note that the set $A_k^{e_i}$ depends on $j_0$ resp. $s(j_0)$, too.  The decomposition \eqref{eq:dec} now can be written in the following way: 

\begin{align} \label{eq:dec-1}
\begin{split}
\E^{x}\left[\1_{\{\tau\le t/2\}}P_{t-\tau}f(X_{\tau})\right]
&=\sum_{k=0}^{\infty}\Phi(k)
=\sum_{k=1}^{\infty}\left(\sum_{i=1}^{d}\Phi^i(k)\right)+\Phi(0) \\
&= \sum_{k=1}^{\infty}\Phi^d(k) + \sum_{k=1}^{\infty}\Phi^{d-1}(k)  +\ldots+   \sum_{k=1}^{\infty}\Phi^1(k)  + \Phi(0) 
\end{split}
\end{align}

Next, let us quantify the position of the random point $X_\tau$, which the process jumps to when leaving the ball $B(x_0, s(j_0))$. Recall that for $i=1,2,\ldots, d$, in \autoref{def:theta_and_R}, we have set $R_i=2^{\lengthR_i}\rho$, which implies 
\begin{align*}
\tfrac{5}{4}R_i\le  |x_0^{i}-y_0^{i}| < \tfrac{10}{4}R_i\,.
\end{align*}

\begin{lemma}\label{lem:jAn} Recall that $x_0, y_0 \in \R^d$ satisfy condition \ref{eq:case-R-i0} for some $i_0$ and that $s(j_0) = \frac{R_{j_0}}{8}$ for  $j_0\in \{i_0,\ldots, d\}$. The following holds true:
	\begin{align*}
&	\bigcup\limits_{u \in B(x_0,s(j_0))} \{ u +  h e_i | \, h \in \R\}\subset y_0+ \Big(\bigotimes_{j=1}^{i-1}\mathcal{J}_{\lengthR_{j_0}}\times \R\times\bigotimes_{j=i+1}^{j_0-1}\mathcal{J}_{\lengthR_{j_0}}\times\bigotimes_{j=j_0}^{d} \mathcal{I}_{\lengthR_j}
\Big) 
&\mbox{ if }i< j_0\,,\\
&	\bigcup\limits_{u \in B(x_0,s(j_0))} \{ u + h e_i | \, h \in \R\}\subset y_0+ \Big(\bigotimes_{j=1}^{j_0-1}\mathcal{J}_{\lengthR_{j_0}}\times\bigotimes_{j=j_0}^{i-1} \mathcal{I}_{\lengthR_j}
\times\R\times\bigotimes_{j=i+1}^{d} \mathcal{I}_{\lengthR_j}
\Big)
&\mbox{ if }i\ge j_0\,,
	\end{align*}
	where $\mathcal{I}_{\lengthR_j}:= \pm[2^{\lengthR_j}\rho, 2^{\lengthR_j+2}\rho)$, $\mathcal{J}_{\lengthR_{j_0}}:= \pm[0, 2^{\lengthR_{j_0}+2}\rho)$, and $\theta_{j_0}, \ldots, \theta_{d} \in \N$. 
\end{lemma}

\medskip

\begin{proof}
The set on the left-hand side describes all possible points that the jump process can jump to when leaving the set $B(x_0, s(j_0))$ by a jump in the $i$-th coordinate direction. In the coordinate direction $k$ for $k \leq j_0$ it might happen that the ball $B(x_0, s(j_0))$ intersects a coordinate axis.

\medskip 

Take 
$z\in \bigcup\limits_{u \in B(x_0,s(j_0))} \{ u + h e_i | \, h \in \R\}$ for $i\in \{1, \ldots, d\}$.
For any $j\in \{j_0,\ldots, d\}\backslash \{i\}$, since $R_j\ge R_{j_0}$ and $u^j=z^j$, it holds that 
$$R_j\le \tfrac{5}{4}R_j-s(j_0)\le |x_0^j-y_0^j|-|x_0^j-z^j|\le |z^j-y_0^j|	
\le |x_0^j-y_0^j|+|x_0^j-z^j|\le\tfrac{10}{4}R_j+s(j_0)< 4 R_j,$$
and therefore $|z^j-y_0^j|\in [2^{\lengthR_j}\rho , 2^{\lengthR_j+2}\rho)$ with $\lengthR_j\in \N$.
For any $j\in \{1,\ldots, j_0-1\}\backslash \{i\}$, since $R_j< R_{j_0}$ and $u^j=z^j$,
$$0\le |z^j-y_0^j|	\le |x_0^j-y_0^j|+|x_0^j-z^j|\le\tfrac{10}{4}R_j+s(j_0)< 4 R_{j_0},$$
and  $|z^j-y_0^j|\in [0, 2^{\lengthR_{j_0}+2}\rho)$ with $\lengthR_{j_0}\in\N$. The proof of the lemma is complete.
\end{proof}

\medskip

\begin{remark}\label{rem:nervig}
	\autoref{lem:jAn} implies the following observation for $k \in \N_0$ and $i \in \{1, \ldots, d\}$:
	\begin{align}
	A^{e_i}_k \ne \emptyset \qquad  &\Longrightarrow 
	\begin{cases}
	k= 0\,\mbox{ or }\, k \geq \sum\limits_{j\in\{j_0,\ldots,d\}} \lengthR_j \quad &\text{ if } i< j_0\,, \\
	k= 0\,\mbox{ or }\, k \geq 
	\sum\limits_{j\in\{j_0,\ldots,d\}\backslash\{i\} }\lengthR_j &\text{ if } i\ge j_0 \,. 
	\end{cases}
	\end{align}
	More precisely, for $z\in A_{k}^{e_i}$, $k\ge 1$, there exists $\gamma\in\N_0^d$ such that $z\in A_{k,\gamma}$, so that  $|z^j-y_0^j|\in [2^{\lengthR_j}\rho,  2^{\lengthR_j+2}\rho)\cap [ 2^{\gamma_j}\rho,  2^{\gamma_j+1}\rho)$ for $j\in\{j_0, \ldots, d\}\backslash\{i\}$ by \autoref{lem:jAn}. Since $\lengthR_j, \gamma_j\in\Z$,  $\gamma_j$ is one of $\theta_j $ or $\theta_j+1$ for $j\in\{j_0, \ldots, d\}\backslash\{i\}$. Therefore
	$$k=\sum\limits_{j\in\{1,\ldots,d\}} \gamma_j\geq	 \sum\limits_{j\in\{j_0,\ldots,d\}} \gamma_j\geq \sum\limits_{j\in\{j_0,\ldots,d\}} \lengthR_j \quad \text{ if } i< j_0,$$
	and $ k\geq	 \sum\limits_{j\in\{j_0,\ldots,d\}\backslash\{i\}} \gamma_j\geq \sum\limits_{j\in\{j_0,\ldots,d\}\backslash\{i\}} \lengthR_j$ if $i\ge j_0 $.
\end{remark}

\begin{lemma}\label{lem:zyi}
(i) For all $k\in\N$, $i\in\{1, \ldots, d\}$ and $z\in A_{k}^{e_i}$ there is $\gamma\in \N_0^d$ with $\sum_{i=1}^d \gamma_i = k$ such that for all $y\in B(y_0, \tfrac{\rho}{8})$ and $j\in \{1, \ldots, d\}$
\begin{align}
|z^j-y^j|\in [2^{\gamma_j-1}\rho,  2^{\gamma_j+2}\rho)   \,. \label{eq:zyjg}
\end{align}
(ii) For all $k\in\N_0$, $i\in\{1, \ldots, d\}$, $z\in A_{k}^{e_i}$, $y\in B(y_0, \tfrac{\rho}{8})$ and all $j\in \{1, \ldots, d\}$
\begin{alignat}{2}
&|z^j-y^j|\in [0,  2^{\lengthR_{j_0}+3}\rho) 
\qquad &&\mbox{ if } j\in \{1, \ldots, j_0-1\}\backslash\{i\},\label{eq:zyj0}\\
&|z^j-y^j|\in [ 2^{\lengthR_j -1}\rho,  2^{\lengthR_j+3}\rho)	\quad &&\mbox{ if } j\in \{ j_0\ldots, d\}\backslash\{i\} \,. \label{eq:zyj}
\end{alignat}	
\end{lemma}

\begin{proof}
	(i) Let $k\in\N$, $i\in\{1, \ldots, d\}$ and $z\in A_{k}^{e_i}$. Choose $\gamma\in\N_0^d$ such that $ z\in A_{k,\gamma}$. Then  for all $y\in B(y_0, \tfrac{\rho}{8})$ and $j\in\{1, \ldots, d\}$,
	\begin{align*}
	2^{\gamma_j-1}\rho\le 2^{\gamma_j}\rho-\rho/8
	&\le  |z^j-y_0^j|-|y_0^j-y^j|\le |z^j-y^j|
	\nn\\
	&\le |z^j-y_0^j|+|y_0^j-y^j|\le 2^{\gamma_j+1}\rho+\rho/8\le 2^{\gamma_j+2}\rho \,,
	\end{align*}
	which proves \eqref{eq:zyjg}. (ii) Let $k\in\N_0$, $i\in\{1, \ldots, d\}$, $z\in A_{k}^{e_i}$, $y\in B(y_0, \tfrac{\rho}{8})$ and $j\in \{1, \ldots, d\}$. If $j\in\{1,\ldots, j_0-1\}\backslash\{i\}$,  by \autoref{lem:jAn}
	\begin{align*}
	0\le |z^j-y^j|\le |z^j-y_0^j|+|y_0^j-y^j|\le 2^{\lengthR_{j_0}+2}\rho+\rho/8\le 2^{\lengthR_{j_0}+3}\rho \,.
	\end{align*}
	If $j\in\{j_0, \ldots, d\}\backslash\{i\}$, again by \autoref{lem:jAn},
	\begin{align*}
	2^{\lengthR_j-1}\rho\le 2^{\lengthR_j}\rho-\rho/8
	&\le  |z^j-y_0^j|-|y_0^j-y^j|\le |z^j-y^j|\nn\\
	&\le |z^j-y_0^j|+|y_0^j-y^j|\le 2^{\lengthR_j+2}\rho+\rho/8\le 2^{\lengthR_j+3}\rho. 
	\end{align*}
	This proves \eqref{eq:zyj}  and \eqref{eq:zyj0}.
\end{proof}

\medskip

Using \autoref{rem:nervig}, we can refine the decomposition \eqref{eq:dec-1} further.  For $a,b \in \N$, we set
\begin{align}
\TT{a}{b}:=\Tj{a}{b}\qquad\mbox{ and } \qquad \gam{a}{b}:=\gj{a}{b} \,. \label{def:TT-gam}
\end{align}
Note that we will make use of the abbreviation $\gam{a}{b}$ only much further below.
The following decomposition will be the main starting point in our proof.

\begin{align} \label{eq:dec-main} 
\begin{split}
&\E^{x}\left[\1_{\{\tau\le t/2\}}P_{t-\tau}f(X_{\tau})\right] \\
&=  \sum_{k=\TT{j_0}{d}-\lengthR_d}^{\infty}\Phi^d(k) + \sum_{k=\TT{j_0}{d}-\lengthR_{d-1}}^{\infty}\Phi^{d-1}(k) + \ldots + \sum_{k=\TT{j_0}{d}-\lengthR_{j_0}}^{\infty}\Phi^{j_0}(k)\\
& \qquad +\sum_{k=\TT{j_0}{d}}^{\infty}\Big(\Phi^{j_0-1}(k)+\ldots+\Phi^1(k)\Big)+\Phi(0) \\
&=
\sum_{i={j_0}}^d\cSi+\sum_{i={1}}^{j_0-1} \cTi+\Phi(0)
\end{split}
\end{align}

where $\cSi:=\sum_{k=\TT{j_0}{d}-\lengthR_i}^{\infty}\Phi^i(k)$
and $\cTi:=\sum_{k=\TT{j_0}{d}}^{\infty}\Phi^i(k)$.

\medskip

Before we continue, let us recall the setting that we are going to use for the rest of \autoref{sec:proof-propo}:

\medskip

\emph{\parbox{\textwidth}{
\begin{align*}
& \text{We assume } l\in \{0, 1,\ldots, d-1\}, \quad i_0\in\{1, \ldots, d-l\}, \quad j_0 \in \{i_0, \ldots,  d - l\}\,, \\
& x_0, y_0 \in \R^d \text{ satisfying } \mathcal{R}(i_0), \quad  s(j_0)=\tfrac{R_{j_0}}{8}, \quad \tau=\tau_{B(x_0, s(j_0))} \,. 
\end{align*}
}}

\medskip

Note that, because of \autoref{prop:xu2}, for $x\in B(x_0, \tfrac{\rho}{8})$ and $\tau=\tau_{B(x_0, s(j_0))}$,
\begin{align}\label{eq:Plg}
\Pp^x(\tau\le t/2, X_{\tau}\in A_k^{e_i})\le \Pp^x(\tau\le t/2)\le c {t}R_{j_0}^{-\alpha}.
\end{align}

\medskip

Our 
aim is to estimate $\Phi(0)$, $\cSi$, $i\in\{j_0, \ldots, d\}$ and $\cTi$, $i\in\{1,\ldots,j_0-1\}$.

Given $l\in \{0,\ldots, d-1\}$, consider $j_0\in\{i_0,\ldots,d-l\}$ and let
\begin{align}
F_{j_0}(l):=\prod_{j=j_0}^{d-l}\left(\tfrac{t}{R_j^\alpha}\right)^q\prod_{j=d-l+1}^d\left(\tfrac{t}{R_j^{\alpha}}\right)^{1+\alpha^{-1}}\,.\label{eq:prodlg}
\end{align}
Since $\tfrac{t}{R_j^\alpha}=2^{-\lengthR_j\alpha}$, note that $F_{j_0}(l)=2^{-\TT{j_0}{d-l}\alpha q}\cdot 2^{-\TT{d-l+1}{d}\alpha(1+\alpha^{-1})}$.

\medskip

\subsection*{\bf Estimates of \texorpdfstring{$\Phi(0)$}{F0}.}

We will apply the decomposition $\Phi(0)=\sum_{i=1}^d\Phi^i(0)$. Let us fix $i \in \{1,\ldots,d\}$. Our first goal is the derivation of an upper bound for 
\begin{align}\label{eq:Pf-z}
P_{t-\tau}f(z)= \int_{B(y_0, \frac{\rho}{8})} p_{t-\tau}(z, y) f(y) dy 
\end{align}
for $z\in A_0^{e_i}$, $t/2\le t-\tau\le t$ and  $f$ is a non-negative Borel function on $\Rd$ supported in $B(y_0, \tfrac{\rho}{8})$. The goal is achieved once we have proved \eqref{eq:Hlq0g3}, \eqref{eq:Hlq0g1} and \eqref{eq:Hlq0g2}.

\medskip

Fix $y\in B(y_0, \rho/8)$. 
Note that trivially $\left(\tfrac{t}{|z^j-y^j|^\alpha}\wedge 1\right)\le 1$ for any index $j$. In the following we make use of \eqref{eq:zyj0} and \eqref{eq:zyj}, i.e.,  
\begin{align*}
\begin{cases}
|z^j-y^j|\in [0, 2^{\lengthR_{j_0}+3}\rho)&\mbox{ if } j\in \{1, \ldots, j_0-1\}\backslash \{i\},\\
|z^j-y^j|\in [2^{\lengthR_{j}-1}\rho, 2^{\lengthR_{j}+3}\rho)&\mbox{ if } j\in \{j_0, \ldots, d\}\backslash \{i\} \,.
\end{cases}
\end{align*}

Set $r_j:=|z^j-y^j|$ for $j\in \{1, \ldots, d\}$. We will estimate $p_{t-\tau}(z, y)$ in dependence of the size of $r_i$. Given $r_i$ we choose $m \in \{0, \ldots, \lengthR_{i}\}$ as follows:

\begin{enumerate}
	\item We choose $m=0$ if $r_i \leq \rho$, which implies $r_i \leq r_{d-l}$.
	\item We choose $m = \lengthR_{i}$ if $r_i > \rho 2^{\lengthR_i}$.
	\item We choose $m\in\{1, \ldots, \lengthR_i - 1\}$ such that $\rho 2^{m-1} < r_i \leq \rho 2^{m+1}$ if  $\rho < r_i \leq \rho 2^{\lengthR_i}$.
\end{enumerate}

\medskip

We consider different cases for the size of the index $i$.

\subsubsection*{Case 1: $d-l < i$} Condition $\HHq{q}{l}$
 (to be precise $\HHq{q}{l}'$)
yields for $t/2\le t-\tau\le t$
\begin{align} \label{eq:pzy-case1}
\begin{split}
p_{t-\tau}(z, y)
\le \ c t^{-d/\alpha}
&\prod_{j\in \{1, \ldots, d-l-1\}}\left(\frac{t}{r_j^\alpha}\wedge 1\right)^q
\prod_{j\in \{d-l+1, \ldots, d\}\backslash\{i\}}\left(\frac{t}{r_j^\alpha}\wedge 1\right)^{1+\alpha^{-1}} \\
&\qquad
\cdot
 \left(\frac{t}{(r_{d-l}\wedge r_i)^\alpha}\wedge 1\right)^q 
	\left(\frac{t}{(r_{d-l}\vee r_i)^\alpha}\wedge 1\right)^{1+\alpha^{-1}}\,.
\end{split}
\end{align}

We want to estimate the term on the right-hand side in \eqref{eq:pzy-case1} from above. The case $m=0$ corresponds to the case $r_i \leq r_{d-l}$. The case $m = \lengthR_{i}$ corresponds to $r_i > \rho 2^{\lengthR_i}$, which implies $2^{3} r_i > r_{d-l}$ in this case. The above choice of $m \in \{1, \ldots, \lengthR_{i}-1\}$ corresponds to $\rho 2^{m} < r_i \leq \rho 2^{m+1}$. Thus, for $m < \lengthR_{d-l}$ we have $r_{d-l} \geq \rho 2^{\theta_{d-l}-1} \geq \rho 2^{m} \geq \tfrac12 r_i$, and for $m \geq \lengthR_{d-l}$ we have $r_{d-l} < \rho 2^{\theta_{d-l}+3} \leq 8 \rho 2^{m} \leq 8 r_i$. We use these inequalities to estimate $r_{d-l}\wedge r_i$ and $r_{d-l}\vee r_i$ in \eqref{eq:pzy-case1} and obtain

\begin{align}\label{eq:n11-12-new}
p_{t-\tau}(z, y)
\le  \ ct^{-d/\alpha}
&\begin{cases}
2^{-\left(m+ \sum_{j=j_0}^{d-l-1}\lengthR_j\right)\alpha q}
2^{\left( \lengthR_{i} - \left(\sum_{j=d-l}^{d}\lengthR_j\right) \right) \alpha (1+\alpha^{-1})}
&\text{ if } 0 \leq m < \lengthR_{d-l}\,,\\
2^{-\sum_{j=j_0}^{d-l}\lengthR_j\alpha q}
2^{\left(\lengthR_{i} -  m - \left( \sum_{j=d-l+1}^{d}\lengthR_j\right) \right)\alpha (1+\alpha^{-1})} 
&\text{ if }  \lengthR_{d-l}\le m \le \lengthR_{i} \,, \\
\end{cases}
\end{align}
where one might need to change the constant $c$ appropriately. Equivalently, we write  
\begin{align}\label{eq:n11-12-new-new}
p_{t-\tau}(z, y)
&\le  \ ct^{-d/\alpha} 2^{-\left(\sum_{j=j_0}^{d-l-1}\lengthR_j \right) \alpha q} 2^{\left( \lengthR_{i} - \left(\sum_{j=d-l+1}^{d}\lengthR_j\right) \right) \alpha (1+\alpha^{-1})} \nonumber \\
& \qquad \times 
\begin{cases}
2^{- \alpha \left( m q + \lengthR_{d-l} (1+\alpha^{-1})\right)}
&\text{ if } 0 \leq m < \lengthR_{d-l}\,,\\
2^{-\alpha \left(\lengthR_{d-l}  q  + m (1+\alpha^{-1})\right) } 
&\text{ if }  \lengthR_{d-l}\le m \le \lengthR_{i} \,, \\
\end{cases}
\end{align}

\medskip

\subsubsection*{Case 2: $j_0 \leq i \le d-l$} 

Here we deduce from $\HHq{q}{l}$ (to be precise $\HHq{q}{l}'$) for $t/2\le t-\tau\le t$ the inequality 

\begin{align}\label{eq:pzy-case2}
\begin{split}
p_{t-\tau}(z, y)
\le c t^{-d/\alpha}&\prod_{j\in \{1, \ldots, d-l\}\backslash\{i\}}\left(\frac{t}{r_j^\alpha}\wedge 1\right)^q
\prod_{j\in \{d-l+2, \ldots, d\}}\left(\frac{t}{r_j^\alpha}\wedge 1\right)^{1+\alpha^{-1}}
\\
&\qquad\cdot
 \left(\frac{t}{(r_{d-l+1}\wedge r_i)^\alpha}\wedge 1\right)^q
	\left(\frac{t}{(r_{d-l+1}\vee r_i)^\alpha}\wedge 1\right)^{1+\alpha^{-1}}.
\end{split}
\end{align}

Again, we estimate $p_{t-\tau}(z, y)$ from above in dependence of the size of $r_i$ resp. the choice of $m$. We obtain
\begin{align}\label{eq:n21-22-new}
p_{t-\tau}(z, y)
\le  \ ct^{-d/\alpha}
&
2^{-\alpha q \left(-\lengthR_i+m + \sum_{j=j_0}^{d-l}\lengthR_j\right)}
2^{-\alpha (1+\alpha^{-1})\sum_{j=d-l+1}^{d}\lengthR_j}
&\text{ if } 0 \leq m \leq \lengthR_{i}\,.
\end{align}

\medskip

\subsubsection*{Case 3: $i< j_0 $} 

In this case we again use \eqref{eq:pzy-case2} in order to obtain 
\begin{align}\label{eq:n3}
p_{t-\tau}(z, y)\le &\ c t^{-d/\alpha}
2^{-\sum_{j=j_0}^{d-l}\lengthR_j\alpha q}
2^{-\sum_{j=d-l+1}^{d}\lengthR_j\alpha (1+\alpha^{-1})} \quad \text{ if } 0 \leq m \leq \lengthR_{i}\,.
\end{align}

\medskip

\begin{remark*}
We will apply the aforementioned upper bounds for $p_{t-\tau}(z, y)$ in the three cases for different choices of $m \in \{0, \ldots, \lengthR_{i}\}$. Note that these bounds are comparable for different choices $m_1, m_2 \in \{0, \ldots, \lengthR_{i}\}$ with $|m_2-m_1| \leq 6$ or any other fixed number. This observation will used without further mentioning. 
\end{remark*}

\medskip

We can now approach our first goal, the upper bound of $P_{t-\tau}f(z)$ in \eqref{eq:Pf-z} for $t/2\le t-\tau\le t$ and $z\in A_0^{e_i}$.  Let us consider four different cases:

\subsubsection*{Case (i): $i<j_0$} The case is simple because \eqref{eq:n3} implies 
\begin{align}\label{eq:Hlq0g3}
P_{t-\tau}f(z)
\le \, c\, t^{-d/\alpha} \|f\|_1F_{j_0}(l).
\end{align}

\subsubsection*{Case (ii): $i\ge j_0$ and $|z^i-y_0^i|< \frac12 \rho$} In this case $r_i = |z^i-y^i| \leq |z^i-y_0^i| + |y_0^i-y^i| \leq \rho$ for $y \in B(y_0, \rho/8)$ and one can apply the case $m=0$ in \eqref{eq:n11-12-new} and \eqref{eq:n21-22-new}. We obtain
\begin{align}
P_{t-\tau}f(z)
\le \, c\, t^{-d/\alpha} \|f\|_1F_{j_0}(l)
\begin{cases}
2^{\lengthR_{d-l}\alpha q}2^{(-\lengthR_{d-l}+\lengthR_i)\alpha (1+\alpha^{-1})}
&\mbox{if $d-l< i$}\, ,\\
2^{\lengthR_i\alpha q}
&\mbox{if $ j_0\le i< d-l$}\,.
\end{cases}\label{eq:Hlq0g1}
\end{align} 

\subsubsection*{Case (iii): $i\ge j_0$, $\frac12 \rho\le |z^i-y_0^i|< 2^{\lengthR_i +1}\rho$} Given $y \in B(y_0, \rho/8)$ and $r_i = |z^i-y^i|$ there is $m\in\{1, \ldots, \lengthR_i-1\}$ such that $2^{m-3} \rho\le r_i <  2^{m+3}\rho$. We apply the corresponding case in \eqref{eq:n11-12-new} and \eqref{eq:n21-22-new} together with the remark above. We deduce
\begin{align}\label{eq:Hlq0g2}
&P_{t-\tau}f(z)=\int_{B(y_0, \tfrac{\rho}{8})} p_{t-\tau}(z, y)f(y) dy\\
\le &\, c\, t^{-d/\alpha} \|f\|_1F_{j_0}(l)
\begin{cases}
2^{-(m-\lengthR_{d-l})\alpha q}2^{-(\lengthR_{d-l}-\lengthR_i)\alpha (1+\alpha^{-1})}
&\mbox{if $d-l< i$ and $m< \lengthR_{d-l}$},\\
2^{-(m-\lengthR_i)\alpha(1+\alpha^{-1})}&\mbox{if $d-l< i$ and $\lengthR_{d-l}\le m< \lengthR_i$},\\
2^{-(m-\lengthR_i)\alpha q}
&\mbox{if $ j_0\le i\le  d-l$ and $1\le  m< \lengthR_i$} \,.
\end{cases}\nn
\end{align}

\medskip

\subsubsection*{Case (iv): $i\ge j_0$, $|z^i-y_0^i|\ge 2^{\lengthR_i +1}\rho$} In this case $r_i = |z^i-y^i| \geq |z^i-y_0^i| - |y_0^i-y^i| \geq 2^{\lengthR_i +1}\rho - \frac18 \rho \geq 2^{\lengthR_i}\rho$ for $y \in B(y_0, \rho/8)$. Thus we can apply  
the case $m=\theta_i$ and conclude \eqref{eq:Hlq0g3} from \eqref{eq:n11-12-new}  and \eqref{eq:n21-22-new}.

\medskip

Finally, we can estimate $\Phi(0)=\sum_{i=1}^d\Phi^i(0)$, which is an important step in the estimate of $\E^{x}\left[\1_{\{\tau\le t/2\}}P_{t-\tau}f(X_{\tau})\right]$ in \eqref{eq:dec-main}. Let us fix $i \in \{1, \ldots, d\}$ and estimate $\Phi^i(0)$. The ideas will later be used in an analogous manner in order to estimate $\Phi^i(k)$ for $k \geq 1$.

\medskip

For $i < j_0$, Case (i) and \eqref{eq:Plg} above implies
\begin{align}\label{eq:s02g-part1}
\E^{x} & \left[  \1_{\{\tau\le t/2\}} \1_{\{X_{\tau}\in A_0^{e_i} \}} P_{t-\tau}f(X_{\tau}) \right] \nn\\
&\quad \le \sup_{z\in A_0^{e_i}}P_{t-\tau}f(z)\cdot
\E^x\left[ \1_{\{\tau\le t/2\}} \1_{\{X_{\tau}\in A_0^{e_i} \}} \right]\nn\\
& \quad \le\, c t^{-d/\alpha} \|f\|_1 F_{j_0}(l)
2^{-\lengthR_{j_0}\alpha}
\le \, c t^{-d/\alpha} \|f\|_1 F_{j_0+1}(l)  2^{-\lengthR_{j_0}\alpha(1+q)}.
\end{align}

Recall $F_{d-l+1}(l) = \prod_{j=d-l+1}^d\left(\tfrac{t}{R_j^{\alpha}}\right)^{1+\alpha^{-1}}$ because of $\prod_{j=d-l+1}^{d-l} (\ldots) = 1$ by definition.

\medskip

Now, let us assume $i \geq j_0$. 

\medskip

First, we observe that due to Case (iv) and \eqref{eq:Plg} above, we have
\begin{align}\label{eq:s02g-part 2}
\E^{x}&\left[\1_{\{\tau\le t/2\}}\1_{\{X_{\tau}\in A_0^{e_i}, |X_{\tau^i-}-y_0^i|\ge 2^{\lengthR_i +1}\rho\}} P_{t-\tau}f(X_{\tau})\right] \nn\\
&\quad \le\,c t^{-d/\alpha} \|f\|_1 F_{j_0}(l)
2^{-\lengthR_{j_0}\alpha}
\le \, c t^{-d/\alpha} \|f\|_1 F_{j_0+1}(l)  2^{-\lengthR_{j_0}\alpha(1+q)}.
\end{align} 

For the remaining two cases we need some auxiliary estimate. The idea is to decompose the event $X_{\tau}\in A_0^{e_i}$ into several disjoint events in dependence on the position of $X^i_{\tau-}$, i.e., the point just before the jump out of $B(x_0, s(j_0))$. 
 We define $ I^i_{0, 0}:=\{\ell \in \R:|\ell - y_0^i| \in [0, 2\rho) \}$ and $ I^i_{0, m}:=\{\ell \in \R:|\ell - y_0^i| \in  [2^{m-1}\rho, 2^{m+1}\rho)\}$  for $ m\in \{1, \ldots, \lengthR_i-1\}$. Then, by the L\'{e}vy system formula
\begin{align}\label{eq:LSIg0}
\E^x & \left[ \1_{\{\tau\le t/2\}} \1_{\{X_{\tau}\in A_0^{e_i}, 2^{m-1}\rho\le |X_{\tau-}^i-y_0^i|< 2^{m+1}\rho \}} \right] \nn \\ 
&\leq \E^x\left[\int_{0}^{t/2\wedge \tau}\int_{ I^i_{0, m}}
\frac{1}{|X_{s-}^i-\ell |^{1+\alpha}} d \ell d s \right]\le \frac{c t}{R_i^{1+\alpha}}\cdot
2^{m}\rho
\end{align}
The estimate \eqref{eq:LSIg0} follow easily by considering $w\in B(x_0, s(j_0))$ and $z\in \Rd$ satisfying $|z^i-y_0^i|\in [2^{\gamma_i}\rho, 2^{\gamma_i+1}\rho)$ with $\gamma_i < \lengthR_i$ or $|z^i-y_0^i| \in [0,  2\rho) \cup [2^{m-1}\rho, 2^{m+1}\rho)$ for some $m < \lengthR_i$. Then for $i\ge {j_0}$, 
\begin{align*} 
|w^i-z^i|&\ge |x_0^i-y_0^i|-|w^i-x_0^i|-|z^i-y_0^i| \ge 5R_i/4-R_i/8-2^{\lengthR_i}\rho= R_i/8 \,,
\end{align*}
which implies the second inequality of \eqref{eq:LSIg0}.

\medskip

Now, we can proceed with Case (ii). 
By \eqref{eq:Hlq0g1} and \eqref{eq:LSIg0}, we obtain 
\begin{align}\label{eq:s01g1}
\E^{x} & \left[\1_{\{\tau\le t/2\}}\1_{\{X_{\tau}\in A_0^{e_i},\ |X_{\tau^i-}-y_0^i|\le \rho/2\}}P_{t-\tau}f(X_{\tau})\right]\nn\\
& \le \, c t^{-d/\alpha} \|f\|_1 F_{j_0}(l)
\begin{cases}
 2^{\lengthR_{j_0}\alpha (q-1-\alpha^{-1})}
&\mbox{  if $d-l< i$  }\,,
\\
2^{-\lengthR_i\alpha(-q+1+\alpha^{-1})}&\mbox{  if $ j_0\le i\le  d-l$  }
\end{cases}\nn\\
& \le\, c t^{-d/\alpha} \|f\|_1 F_{j_0+1}(l) 
\begin{cases}
2^{-\lengthR_{j_0} \alpha(1+\alpha^{-1})}&\mbox{  if $d-l< i$  }\,,\\
2^{-\lengthR_{j_0} \alpha(1+\alpha^{-1})}&\mbox{  if $ j_0\le i\le  d-l$ }.
\end{cases}
\end{align}
The first inequality holds in the case $d-l< i$ since $q\le 1+\alpha^{-1}$ and $\lengthR_{j_0}\le\lengthR_{d-l}$. Note that $j_0\in \{i_0,\ldots,  d-l\}$ implies  $\lengthR_{j_0}\le\lengthR_{d-l}$. The last inequality holds in the case $j_0\le i\le  d-l$ because $q\le 1+\alpha^{-1}$ and $\lengthR_{i}\ge \lengthR_{j_0}$.

\medskip

It remains to consider Case (iii). By 
\eqref{eq:Hlq0g2} and \eqref{eq:LSIg0}, 
\begin{align*}
&\E^{x}\left[\1_{\{\tau\le t/2\}}\1_{\{X_{\tau}\in A_0^{e_i}, 2^{m-1}\rho\le |X_\tau^i-y_0^i|< 2^{m+1}\rho\}}P_{t-\tau}f(X_{\tau})\right]\\
\le \,&c t^{-d/\alpha} \|f\|_1 F_{j_0}(l)
\begin{cases}
2^{-m\alpha (q-\alpha^{-1})}2^{-\lengthR_{d-l}\alpha (-q+1+\alpha^{-1})}
\qquad&\mbox{if $d-l< i$ and $m< \lengthR_{d-l}$}\,,
\\
2^{-m\alpha} 
&\mbox{if $d-l< i$ and $ \lengthR_{d-l}\le m <\lengthR_i$}\,,
\\
2^{-m\alpha (q-\alpha^{-1})}2^{-\lengthR_{i}\alpha (-q+1+\alpha^{-1})}&\mbox{if $ j_0\le i\le  d-l$ and $1\le  m< \lengthR_i$},
\end{cases}\nn\\
\le\,& c t^{-d/\alpha} \|f\|_1 F_{j_0+1}(l) 
\begin{cases}
2^{-m\alpha (q-\alpha^{-1})}2^{-\lengthR_{j_0}\alpha q}2^{-\lengthR_{d-l}\alpha (-q+1+\alpha^{-1})}
&\mbox{if $d-l< i$ and $m< \lengthR_{d-l}$}\,,
\\
2^{-\lengthR_{j_0}\alpha (1+q)} 
&\mbox{if $d-l< i$ and $ \lengthR_{d-l}\le m <\lengthR_i$}\,,
\\
2^{-m\alpha (q-\alpha^{-1})}2^{-\lengthR_{j_0}\alpha q}2^{-\lengthR_{i}\alpha (-q+1+\alpha^{-1})}&\mbox{if $ j_0\le i\le  d-l$ and $1\le  m< \lengthR_i$}\,.
\end{cases}\nn
\end{align*}
Since 
$2^{-m\alpha (q-\alpha^{-1})}\le 1 $ for $q>\alpha^{-1}$ and
$2^{-m\alpha (q-\alpha^{-1})}\le 2^{-(\lengthR_{i}\wedge \lengthR_{d-l})\alpha (q-\alpha^{-1})}$ for $q<\alpha^{-1}$ and $m<\lengthR_{i}\wedge \lengthR_{d-l}$,
\begin{align}\label{eq:s01g4}
&\E^{x}\left[\1_{\{\tau\le t/2\}}\1_{\{X_{\tau}\in A_0^{e_i}, 2^{m-1}\rho\le |X_\tau^i-y_0^i|< 2^{m+1}\rho\}}P_{t-\tau}f(X_{\tau})\right]\nn\\
&\quad\le\, c t^{-d/\alpha} \|f\|_1 F_{j_0+1}(l) 
\begin{cases}
2^{-\lengthR_{j_0} \alpha(1+q)}&\mbox{ if } q<\alpha^{-1}\,,\\
2^{-\lengthR_{j_0} \alpha(1+\alpha^{-1})}&\mbox{ if } q>\alpha^{-1}.
\end{cases} 
\end{align}

Therefore, by \eqref{eq:s02g-part1} -- \eqref{eq:s01g4}, we obtain 
\begin{align}\label{eq:S0g}
\Phi(0)=\sum_{i=1}^d\Phi^i(0)\le\, c t^{-d/\alpha} \|f\|_1 F_{j_0+1}(l) 
\begin{cases}
2^{-\lengthR_{j_0} \alpha(1+q)}&\mbox{ if } q<\alpha^{-1}\,,\\
2^{-\lengthR_{j_0} \alpha(1+\alpha^{-1})}&\mbox{ if } q>\alpha^{-1}.
\end{cases} 
\end{align}

\subsection*{\bf Estimates of \texorpdfstring{$\cSi:=\sum_{k=\TT{j_0}{d}-\lengthR_i}^{\infty}\Phi^i(k)$ for $i\in \{j_0, \ldots, d\}$}{Si}.}

In the sequel we will make use of the following. Let $x\in B(x_0, \tfrac{\rho}{8})$ and $i\ge {j_0}$. Let us consider two cases. (1) For $k\ge 1$ and $\gamma_i < \theta_i$, we set  
$I_k^i:=\{\ell \in\R:|\ell - y_0^i|\in [2^{\gamma_i}\rho, 2^{\gamma_i+1}\rho)\}$ with $\gamma_i$ as in \autoref{def:D-sets}. In this case,  
\begin{align}\label{eq:LSIg}
&\E^x\left[\int_{0}^{t/2\wedge \tau}\int_{I_k^i}
\frac{1}{|X_s^i-\ell |^{1+\alpha}} d \ell d s \right]\le \frac{c t}{R_i^{1+\alpha}}\cdot
2^{\gamma_i}\rho \,.
\end{align}
The proof of \eqref{eq:LSIg} is analogous to the one of  \eqref{eq:LSIg0}. (2) For $k\ge 1$ and $\gamma_i \ge \theta_i$, we will use \eqref{eq:Plg}.
 
\medskip

Let $l\in\{0,1,\ldots, d-1\}$, $i_0\in\{1,\ldots, d-l\}$ and $j_0\in\{i_0,\ldots, d-l\}$. 
We follow the strategy of the estimate of $\Phi(0)$. 	For $z\in A_k^{e_i}$ and $y\in B(y_0, \rho/8)$, let $r_j:=|z^j-y^j|$ for $j\in \{1, \ldots, d\}$. Note that  $2^{\gamma_{j}-1}\rho\le r_j<2^{\gamma_{j}+2}\rho$ for $j\in \{1, \ldots, j_0-1\}\cup\{i\}$ and $2^{\lengthR_{j}-1}\rho\le r_j<  2^{\lengthR_{j}+3}\rho$ for $j\in \{j_0, \ldots, d\}\backslash\{i\}$ by \eqref{eq:zyjg} and \eqref{eq:zyj}. 
Hence $\HHq{q}{l}$ (to be precise $\HHq{q}{l}'$) 
yields for $t/2\le t-\tau\le t$ and $i\le d-l$,
\begin{align*}
p_{t-\tau}(z, y)
&\le c t^{-d/\alpha}
\prod_{j\in \{1, 2, \ldots, d-l\}}\left(\frac{t}{r_j^\alpha}\wedge 1\right)^q
\prod_{j\in \{d-l+1, \ldots, d\}}\left(\frac{t}{r_j^\alpha}\wedge 1\right)^{1+\alpha^{-1}}\nn\\
&\le c t^{-d/\alpha}
\prod_{j\in \{1, 2, \ldots, d-l\}}\left(\frac{t}{r_j^\alpha}\right)^q
\prod_{j\in \{d-l+1, \ldots, d\}}\left(\frac{t}{r_j^\alpha}\right)^{1+\alpha^{-1}}\\
&\le c t^{-d/\alpha}2^{-(\sum_{j=1}^{j_0-1}\gamma_j+\gamma_i)\alpha q}
2^{-(\sum_{j=j_0}^{d-l}\lengthR_j-\lengthR_i)\alpha q}
2^{-\left(\sum_{j=d-l+1}^{d}\lengthR_j\right)\alpha (1+\alpha^{-1})} \,.
\end{align*}

We note that 
$\gamma=(\gamma_1,\cdots,\gamma_d)\in \N_0^d$ is determined once we fix
$A_k^{e_i}$ and $B(y_0, \rho/8)$. Note that it is independent
of the choice of the elements $z$ and $y$. 

\medskip

The case $i > {d-l}$ can be dealt with similarly. Altogether, by \eqref{eq:zyjg}--\eqref{eq:zyj}, 
$\HHq{q}{l}$ 
yields the following estimate for $t/2\le t-\tau\le t$ and $z\in A_k^{e_i}\cap A_k^\gamma$: 
\begin{align}
&P_{t-\tau}f(z)=\int_{B(y_0, \frac{\rho}{8})} p_{t-\tau}(z, y) f(y) dy\nn\\
&\le c t^{-d/\alpha} \|f\|_1
\begin{cases}
2^{-(\gam{1}{j_0-1}+\TT{j_0}{d-l}+\gamma_i-\lengthR_i)\alpha q}\cdot 2^{-\TT{d-l+1}{d}\alpha(1+\alpha^{-1})}
&\mbox{ if $i\le {d-l}$,}\\
2^{-(\gam{1}{j_0-1}+\TT{j_0}{d-l})\alpha q}2^{-(\TT{d-l+1}{d}+\gamma_i-\lengthR_i)\alpha(1+\alpha^{-1})}
&\mbox{ if $i>d-l$}
\end{cases}\label{eq:Hlqg1}\\
&\le c t^{-d/\alpha} \|f\|_1\cdot F_{j_0}(l)
\begin{cases}
2^{-(\gam{1}{j_0-1}+\gamma_i-\lengthR_i)\alpha q}
&\mbox{ if $\gamma_i< \lengthR_{i}$,}\\
2^{-\gam{1}{j_0-1}\alpha q} 2^{-(\gamma_i-\lengthR_i)\alpha(1+\alpha^{-1})}  &\mbox{  if 
$\gamma_i\ge \lengthR_{i}$}
\end{cases}\label{eq:Hlqg}
\end{align}
where $F_{j_0}(l)$	is defined in \eqref{eq:prodlg}. Note that we have used $\gam{a}{b}$ as defined in \eqref{def:TT-gam}. Here the last inequality is due to the fact 
$-(\gamma_i-\lengthR_i)\alpha q<-(\gamma_i-\lengthR_i)
\alpha(1+\alpha^{-1})$ for $\gamma_i< \lengthR_{i}$ 
and the opposite inequality holds for 
$\gamma_i\ge \lengthR_{i}$ (note that 
$0\le q\le 1+\alpha^{-1}$).

\medskip

With regard to the definition of $A_k$, recall $k=\sum_{j=1}^d \gamma_j$. For $z\in A_k^{e_i}$, there exists $\gamma\in \N_0^d$ such that $z\in A_{k, \gamma}$ with  $\gamma_j=\lengthR_j \mbox{ or }\lengthR_j+1$ for $j\in\{j_0, \ldots, d\}\backslash\{i\}$ (see, \autoref{rem:nervig}). Therefore, 
\begin{align}\label{eq:gig}
\gam{1}{j_0-1}+\TT{j_0}{d}-\lengthR_i+\gamma_i\le k\le \gam{1}{j_0-1}+\TT{j_0}{d}-\lengthR_i+\gamma_i+d.
\end{align}

Now decompose $\cSi$ as follows:
\begin{align*}
\cSi
=
\sum_{k= \TT{j_0}{d}-\lengthR_i}^{\infty} 
	\Phi^i(k)1_{\{\gamma_i< \lengthR_i\}}+\sum_{k= \TT{j_0}{d}-\lengthR_i}^{\infty} 
	\Phi^i(k)1_{\{\gamma_i\ge \lengthR_i\}}
			=:\I+\II.
\end{align*}

By \eqref{eq:LSd}, \eqref{eq:LSIg} and \eqref{eq:Hlqg},
\begin{align*}
&\I=\,
	\sum_{k= \TT{j_0}{d}-\lengthR_i}^{\infty}
	\E^{x}\left[\1_{\{\tau\le t/2\}}\1_{\{X_{\tau}\in A_k^{e_i}\}}P_{t-\tau}f(X_{\tau})\right]1_{\{\gamma_i< \lengthR_i\}}\nn\\
\le \,&c t^{-d/\alpha} \|f\|_1 F_{j_0}(l)
\sum_{k= \TT{j_0}{d}-\lengthR_i}^{\infty} 
2^{-(\gam{1}{j_0-1}+\gamma_i-\lengthR_i)\alpha q} \cdot \E^x\left[\int_{0}^{t/2\wedge \tau}\int_{I_k^i}
\frac{1_{\{\gamma_i< \lengthR_i\}}}{|X_s^i-z^i|^{1+\alpha}}dz^ids\right]\nn\\
\le \,&c t^{-d/\alpha} \|f\|_1 F_{j_0}(l)
\sum_{k= \TT{j_0}{d}-\lengthR_i}^{\infty} 2^{-(\gam{1}{j_0-1}+\gamma_i-\lengthR_i)\alpha q} \cdot2^{\gamma_i} 2^{-\lengthR_i\alpha(1+\alpha^{-1})}1_{\{\gamma_i< \lengthR_i\}}.
\end{align*}

Note that $\gam{1}{j_0-1}+\gamma_i-\lengthR_i \ge k-\TT{j_0}{d} -d$ by the second inequality of \eqref{eq:gig} and 
$\gamma_i-\lengthR_i\le k-\gam{1}{j_0-1}-\TT{j_0}{d}\le k-\TT{j_0}{d}$ by the first inequality of \eqref{eq:gig}.
So when
$\gamma_i< \lengthR_i$,
$$2^{-(\gam{1}{j_0-1}+\gamma_i-\lengthR_i)\alpha q} 2^{\gamma_i-\lengthR_i}2^{-\lengthR_i\alpha}\le 
\begin{cases}
2^{d\alpha q}\cdot  2^{-(k-\TT{j_0}{d})(\alpha q-1)}2^{-\lengthR_i\alpha}\\
2^{d\alpha q}\cdot2^{-(k-\TT{j_0}{d})\alpha q}2^{-\lengthR_i\alpha}.
\end{cases}$$ 
Therefore, 
\begin{align}\label{eq:s1g}
\I\le\, &c t^{-d/\alpha} \|f\|_1 F_{j_0}(l)
\left(\sum_{k= \TT{j_0}{d}-\lengthR_i}^{\TT{j_0}{d}-1}2^{-(k-\TT{j_0}{d})(\alpha q-1)}2^{-\lengthR_i\alpha}+\sum_{k= \TT{j_0}{d}}^{\infty}2^{-(k-\TT{j_0}{d})\alpha q}2^{-\lengthR_i\alpha}\right)\nn\\
\le\, &c t^{-d/\alpha} \|f\|_1 F_{j_0}(l)
\begin{cases}
2^{-\lengthR_i\alpha}&\mbox{ if } q<\alpha^{-1}\,,\\
2^{-\lengthR_i\alpha(-q+\alpha^{-1}+1)}+2^{-\lengthR_i\alpha}&\mbox{ if } q>\alpha^{-1}.
\end{cases}
\end{align}

For $\II$, applying \eqref{eq:Plg}, \eqref{eq:Hlqg} 
and \eqref{eq:gig}, we have that 
\begin{align}\label{eq:s23ag}
\II\le \,& c t^{-d/\alpha} \|f\|_1
F_{j_0}(l)\sum_{k= \TT{j_0}{d}}^{\infty}
2^{-(\gam{1}{j_0-1}+\gamma_i-\lengthR_i)\alpha q} \cdot \Pp^x(\tau\le t/2, X_{\tau}\in A_k^{e_i})1_{\{\gamma_i\ge \lengthR_i\}}\nn\\
\le \,&c t^{-d/\alpha} \|f\|_1 F_{j_0}(l)
\sum_{k= \TT{j_0}{d}}^{\infty} 2^{-(\gam{1}{j_0-1}+\gamma_i-\lengthR_i)\alpha q} \cdot
2^{-\lengthR_{j_0}\alpha}\nn\\
 \le  \,&c t^{-d/\alpha} \|f\|_1   F_{j_0}(l)
\sum_{k=\TT{j_0}{d}}^{\infty}
2^{-(k-\TT{j_0}{d} -d)\alpha q}2^{-\lengthR_{j_0}\alpha}\le\, c t^{-d/\alpha} \|f\|_1 F_{j_0}(l) \cdot  2^{-\lengthR_{j_0}\alpha}.
\end{align}

Since $j\to \theta_j$ is increasing and $q\le 1+\alpha^{-1}$,  \eqref{eq:s1g}--\eqref{eq:s23ag} imply that for any $i\in \{j_0,\ldots, d\}$, 
\begin{align}\label{eq:Sig}
\cSi\le \,&c t^{-d/\alpha} \|f\|_1   F_{j_0}(l)
\begin{cases}
2^{-\lengthR_i\alpha}
+2^{-\lengthR_{j_0}\alpha}\le 2^{-\lengthR_{j_0}\alpha}&\mbox{ if } q<\alpha^{-1}\,,\\
2^{-\lengthR_i\alpha(-q+\alpha^{-1}+1)}
+2^{-\lengthR_{j_0}\alpha}&\mbox{ if } q>\alpha^{-1}
\end{cases}\nn\\
\le \,&c t^{-d/\alpha} \|f\|_1   F_{j_0+1}(l)
\begin{cases}
2^{-\lengthR_{j_0}\alpha(1+q)}&\mbox{ if } q<\alpha^{-1}\,,\\
2^{-\lengthR_{j_0}\alpha(1+\alpha^{-1})}&\mbox{ if } q>\alpha^{-1}.
\end{cases}
\end{align}

\medskip

\subsection*{\bf Estimates of \texorpdfstring{$\cTi$ for $i\in\{1,2,\ldots,j_0-1\}$}{Ti}.} 

Let $l\in\{0,1,\ldots, d-1\}$, $i_0\in\{1,\ldots, d-l\}$ and  $j_0\in\{i_0,\ldots, d-l\}$. Analogous to the proof of \eqref{eq:Hlqg1}, we apply \eqref{eq:zyjg}--\eqref{eq:zyj} together with 
$\HHq{q}{l}$ (to be precise $\HHq{q}{l}'$) in order to prove that for $t/2\le t-\tau\le t$ and $z\in A_k^{e_i}$, 
\begin{align}\label{eq:Hlq1g}
P_{t-\tau}f(z)&=\int_{B(y_0, \frac{\rho}{8})} p_{t-\tau}(z, y) f(y) dy\nn\\
&\le c t^{-d/\alpha} \|f\|_12^{-(\gam{1}{j_0-1}+\TT{j_0}{d-l})\alpha q}
2^{-\TT{d-l+1}{d}\alpha (1+\alpha^{-1})}\nn\\
&= c t^{-d/\alpha} \|f\|_12^{-\gam{1}{j_0-1}\alpha q}
\cdot F_{j_0}(l)
\end{align} 
where $F_{j_0}(l)$	is defined in \eqref{eq:prodlg}. Regarding the definition of $A_k$, recall $k=\sum_{j=1}^d \gamma_j$.
For $z\in A_k^{e_i}$, there exists $\gamma\in \N_0^d$ such that $z\in A_{k, \gamma}$ with  $\gamma_j=\lengthR_j \mbox{ or }\lengthR_j+1$ for $j\in\{j_0, \ldots, d\}$ (see, \autoref{rem:nervig}),
hence 
	\begin{align}\label{eq:gilg}
\gam{1}{j_0-1}+\TT{j_0}{d}\le	k\le\gam{1}{j_0-1}+\TT{j_0}{d}+d.
	\end{align}
Therefore, by \eqref{eq:Plg} and \eqref{eq:Hlq1g} with \eqref{eq:gilg},
\begin{align}\label{eq:Tig}
\cTi&=\sum_{k\ge\TT{j_0}{d}}\E^{x}\left[\1_{\{\tau\le t/2\}}\1_{\{X_{\tau}\in A_k^{e_i}\}}P_{t-\tau}f(X_{\tau})\right]\nn\\
&\le c t^{-d/\alpha} \|f\|_1
F_{j_0}(l)\sum_{k\ge \TT{j_0}{d}}2^{-\gam{1}{j_0-1}\alpha q}\cdot \Pp^x(\tau\le t/2, X_{\tau}\in A_k^{e_i})\nn\\
&\le c t^{-d/\alpha} \|f\|_1
F_{j_0}(l)\sum_{k\ge \TT{j_0}{d}} 2^{-\gam{1}{j_0-1}\alpha q} 2^{-\lengthR_{j_0}\alpha}\nn\\
& \le  c t^{-d/\alpha} \|f\|_1
F_{j_0}(l)\sum_{k\ge \TT{j_0}{d}} 2^{-(k-\TT{j_0}{d} -d)\alpha q} 2^{-\lengthR_{j_0}\alpha}\nn\\
&\le c t^{-d/\alpha} \|f\|_1
F_{j_0}(l)2^{-\lengthR_{j_0}\alpha}
= c t^{-d/\alpha} \|f\|_1
F_{j_0+1}(l)\cdot 2^{-\lengthR_{j_0}\alpha(1+q)}.
\end{align}

\subsection*{\bf Conclusion} 
Finally, we use the estimates \eqref{eq:S0g}, \eqref{eq:Sig} and \eqref{eq:Tig} in the representation \eqref{eq:dec-main}.
Since $\tfrac{t}{R_i^\alpha}=2^{-\lengthR_i\alpha}$ for $i\in \{1,\ldots, d\}$ by \eqref{d:thi_Ri} and
$$ F_{j_0+1}(l)=\prod_{j=j_0+1}^{d-l}
\left(\tfrac{t}{R_j^\alpha}\right)^q\prod_{j=d-l+1}^d\left(\tfrac{t}{R_j^{\alpha}}\right)^{1+\alpha^{-1}},$$
we obtain the upper bound of \eqref{eq:dec-main} as follows:
\begin{align}\label{eq:st}
&\E^{x}\left[\1_{\{\tau\le t/2\}}P_{t-\tau}f(X_{\tau})\right]
=\sum_{i=1}^{j_0-1}\cTi+\sum_{i=j_0}^{d}\cSi+\Phi(0)\nn\\
&\le c t^{-d/\alpha} \|f\|_1    \prod_{j=j_0+1}^{d-l}
\left(\tfrac{t}{R_j^\alpha}\right)^q\prod_{j=d-l+1}^d\left(\tfrac{t}{R_j^{\alpha}}\right)^{1+\alpha^{-1}} \cdot
\begin{cases}
\left(\tfrac{t}{R_{j_0}^\alpha}\right)^{1+q}
&\mbox{ if } q<\alpha^{-1}\\
\left(\tfrac{t}{R_{j_0}^\alpha}\right)^{1+\alpha^{-1}}
&\mbox{ if } q>\alpha^{-1}.
\end{cases}
\end{align}
This proves \autoref{prop:main} since $R_i\asymp |x_0^i-y_0^i|$ for $i\in\{1,\ldots, d\}$, cf. \eqref{d:thi_Ri}.

\section{Appendix}

In this section we provide the proof of the auxiliary result 
\autoref{lem:symmetry}. Its proof makes uses of a simple algebraic observation, which we provide first.

\medskip

\begin{lemma}\label{lem:number-symm}
	Assume $0 \leq q \leq a$, $l\in \{0,\ldots, d\}$ and $z^1, \ldots,  z^d > 0$. Let $\sigma: \{1, \ldots, d\} \to \{1, \ldots, d\}$ denote a permutation satisfying  $z^{\sigma(i)}\ge z^{\sigma(i+1)}$ for every $i\in\{1,\ldots, d-1\}$. Then
	\begin{align}\label{eq:monosig}
	\prod_{i=1}^{d-l}\left(z^{\sigma(i)}  \right)^{q}
	\prod_{i= d-l+1}^d \left(z^{\sigma(i)}  \right) ^a \leq \prod_{i=1}^{d-l}
	\left(z^{i}  \right)^{q}
	\prod_{i= d-l+1}^d \left(z^{i}\right)^a \,.
	\end{align}
\end{lemma}

\begin{proof}
	We prove \eqref{eq:monosig} by induction for $d$. First, consider the case $d=2$.  If $z^1> z^2$, then 
	$\sigma$ is identity and \eqref{eq:monosig} is trivial, so consider the case $z^1\le z^2$, in which case
	$\sigma=(1~2)$.  When $l=0$ or $l=2$, \eqref{eq:monosig} trivially holds with equality, so consider the case 
	$l=1$. In this case, \eqref{eq:monosig} is $\left(z^{2}  \right)^{q} \left(z^{1}  \right)^{a} \leq \left(z^{1}  \right)^{q} \left(z^{2}  \right)^{a}$, which is equivalent to $\left(z^{1}  \right)^{a-q} \leq \left(z^{2}  \right)^{a-q}$. This inequality holds true because the function $x \mapsto x^{a-q}$ is monotone for $x > 0$. We have proved \eqref{eq:monosig} for $d=2$.
	(Note that, by the same proof the above holds for transpositions of any pair of natural numbers $i<j$ with 
	$z^{\sigma(i)}\ge z^{\sigma(j)}$.) 

\medskip
	
	Now assume \eqref{eq:monosig} holds for $d\le m$ and let us prove it for $d=m+1$.  
	Let $\sigma: \{1, \ldots, m+1\} \to \{1, \ldots, m+1\}$ denote a permutation satisfying  $z^{\sigma(i)}\ge z^{\sigma(i+1)}$ for all $i\in\{1,\ldots, m\}$. Set $k=\sigma(1)$. We may assume $k\ne 1$ since otherwise \eqref{eq:monosig} holds by the induction hypothesis. Let $\hat\sigma: \{1, \ldots, k-1, k+1, \ldots, m+1\} \to \{1, \ldots, k-1, k+1, \ldots, m+1\}$
	be such that $\hat\sigma (i)=\sigma (i+1)$ for $1\le i\le k-1$ and $\hat\sigma(i)=\sigma (i)$ for $k+1\le i\le m+1$ if $k\le m$. Also, let $\tau_i=(k-i ~~ k)=(\sigma (1)-i ~~ \sigma (1))$ for $1\le i \le k-1$. Then it holds that 
	$z^{\hat\sigma(i)}\ge z^{\hat\sigma(j)}$ for $i<j$, $z^{\tau_i(k-i)}=z^k\ge z^{k-i}=z^{\tau_i (k)}$ for $1\le i \le k-1$, and that 
	\[
	\sigma=\hat\sigma\circ\tau_{k-1}\circ \cdots \circ\tau_1,
	\]
	where the order of operations is from the right to the left.
	Hence, using the induction hypothesis repeatedly, we obtain \eqref{eq:monosig} for $d=m+1$. 
\end{proof}

\medskip

\begin{proof}[Proof of \autoref{lem:symmetry}]
	The proof consists of two parts. Let us first show the inequality \eqref{eq:Hi-no-permute}. Assume $t > 0$ and $x,y \in \R^d$. Let $\sigma: \{1, \ldots, d\} \to \{1, \ldots, d\}$ denote a permutation such that $|x^{\sigma(i)}-y^{\sigma(i)}|\le |x^{\sigma(i+1)}-y^{\sigma(i+1)}|$ for every $i\in\{1,\ldots, d-1\}$. Assume that \eqref{eq:Hi-permute} holds true. We apply \autoref{lem:number-symm} with $z^i = \frac{t}{|x^i - y^i|^\alpha} \wedge 1$ and $a = 1 + \alpha^{-1}$.
	Then \eqref{eq:Hi-no-permute} follows from \autoref{lem:number-symm}.
	
	\medskip
	
	The second task is to show that $\HHq{q}{l}$ implies \eqref{eq:Hi-permute}. \autoref{rem:constant-Hlq} will be essential in this step. Recall that the jump kernel $J$ satisfies \eqref{eq:J-ellipticity-assum} and the corresponding heat kernel is denoted by $p_t(x,y)$. Assume $t > 0$ and $x_0,y_0 \in \R^d$. We want to show \eqref{eq:Hi-permute} for $t > 0$ and $x_0,y_0 \in \R^d$. Let $\sigma: \{1, \ldots, d\} \to \{1, \ldots, d\}$ denote a permutation such that $|x^{\sigma(i)}_0-y^{\sigma(i)}_0|\le |x^{\sigma(i+1)}_0-y^{\sigma(i+1)}_0|$ for every $i\in\{1,\ldots, d-1\}$. 
	
	\medskip
	
	For $x \in \R^d$,  
allowing abuse of notation denote by 
$\sigma(x)$ the point $(x^{\sigma(1)}, x^{\sigma(2)}, \ldots x^{\sigma(d)})$. Define a new jump kernel $J^\sigma$ by $J^\sigma(x,y)=J(\sigma(x),\sigma(y))$ and a new bilinear form $\mathcal{E}^{\sigma}$ by 
	\begin{align*}
	\mathcal{E}^{\sigma} (u,v) &=\int_{\Rd}\Big(\sum_{i=1}^d\int_{\R}\big(u(x+e^i \tau) - u(x)\big)\big(v(x+e^i \tau) -  v(x)\big) J^\sigma (x,x+e^{i} \tau) \d \tau \Big) \d x \,.
	\end{align*}
	The domain $D =\{u\in L^2(\Rd)| \; \mathcal{E} (u,u)<\infty\}$ stays unchanged. For $\lambda > 0$ we set $\mathcal{E}_\lambda (u,v) = \mathcal{E} (u,v) + \lambda (u,v)$ and define $\mathcal{E}^{\sigma}_\lambda$ analogously. Then for all $f,g \in D$
	\begin{align*}
	\mathcal{E}^\sigma_\lambda(f,g) = \mathcal{E}_\lambda (f \circ \sigma^{-1}, g \circ \sigma^{-1}) \,.
	\end{align*}
	Note that $(\mathcal{E}^\sigma, D)$ is a regular Dirichlet form just like $(\mathcal{E}, D)$. Let us denote the process corresponding to $(\mathcal{E}^\sigma, D)$ by $X^\sigma$, the semigroup by $P^\sigma$ and the corresponding heat kernel by $p^\sigma_t(x,y)$. 	Recall that we intend to show that $\HHq{q}{l}$ implies \eqref{eq:Hi-permute}. The variables $t>0$ and $x_0, y_0$ have been fixed at he beginning of the proof. Let us assume that we can show 
	\begin{align}\label{eq:p-p_sigma}
	p_t(x_0,y_0) = p^\sigma_t(\sigma(x_0),\sigma(y_0)) \,.
	\end{align}
	Because of \autoref{rem:constant-Hlq} and the fact that the tuple $(\sigma(x_0),\sigma(y_0))$ satisfies the required ordering we can apply \eqref{eq:Hi} to $p^\sigma_t$ and the points $\sigma(x_0),\sigma(y_0)$. Thus, the desired estimate in \eqref{eq:Hi-permute} for $x_0, y_0$ would follow. Hence, it is sufficient to prove \eqref{eq:p-p_sigma}. 
	
	\medskip
	
	In order to show \eqref{eq:p-p_sigma} it is sufficient to prove for every non-negative function $f$ and every $x \in \R^d$
	\begin{align}\label{eq:P-P_sigma}
	P^\sigma_t f(x) = P_t (f \circ \sigma) (\sigma^{-1} (x)) \,.
	\end{align}
	Condition \eqref{eq:P-P_sigma} implies for every $x \in \R^d$ and every $f \geq 0$
	\begin{align*}
	\int p^\sigma_t(x,y) f(y) \d y &= P_t^\sigma f(x) = P_t(f \circ \sigma)(\sigma^{-1}(x)) = \int p_t(\sigma^{-1}(x),z) f(\sigma(z)) \d z \\
	&= \int p_t(\sigma^{-1}(x),\sigma^{-1}(y)) f(y) \d y\,,
	\end{align*}
	which proves \eqref{eq:p-p_sigma}. In order to prove \eqref{eq:P-P_sigma} we introduce for $\lambda > 0$ the Green operators $G_\lambda, G^\sigma_\lambda$ in the usual way:
	\begin{align*}
	G_\lambda f(x) = \int\limits_0^\infty e^{-\lambda t} P_t f(x) \d t, \qquad G^\sigma_\lambda f(x) = \int\limits_0^\infty e^{-\lambda t} P^\sigma_t f(x) \d t\,. 
	\end{align*}  
	By the uniqueness of the Laplace 
	transform (note that we know $P_t f$ and $P^\sigma_t f$ are continuous in this case) 
	it is sufficient to prove for every non-negative function $f$ and every $x \in \R^d$
	\begin{align}\label{eq:G-G_sigma}
	G_\lambda^\sigma f(x) = G_\lambda (f \circ \sigma) (\sigma^{-1} (x)) \,.
	\end{align}
	Let $\phi, f$ be non-negative functions and $\lambda > 0$. Applying the reproducing property for $\mathcal{E}_\lambda$, i.e., the identity $\int u v \, \d x = \mathcal{E}_\lambda(G_\lambda u,v)$ resp. the analogous identity for $\mathcal{E}^\sigma_\lambda$, we obtain 
	\begin{align*}
	\int &\phi(x) G_\lambda(f \circ \sigma)(\sigma^{-1}(x)) \d x = \mathcal{E}_\lambda^\sigma \big( G_\lambda^\sigma \phi, G_\lambda(f \circ \sigma)(\sigma^{-1}(\cdot)) \big)  \\
	&= \mathcal{E}_\lambda \big( G_\lambda^\sigma \phi(\sigma^{-1}(\cdot)), G_\lambda(f \circ \sigma)(\sigma^{-1} \circ \sigma^{-1}(\cdot)) \big)  \\
	&= \int G_\lambda^\sigma \phi(\sigma^{-1}(x)) \big(f \circ \sigma \big) (\sigma^{-1}( \sigma^{-1}(x))) \, \d x \\
	&= \int G_\lambda^\sigma \phi(z)  f(z) \, \d z  = \int \phi(x) G_\lambda^\sigma f(x) \d x\,,
	\end{align*} 
	which proves \eqref{eq:G-G_sigma}. Note that the main ingredient in this part of the proof is the rotational invariant of the Lebesgue measure. The proof is complete.
\end{proof}

\begin{small}

\end{small}

\enlargethispage{\baselineskip}
	
\end{document}